%% file: main.tex
\documentclass[review, hidelinks, onefignum, onetabnum]{siamart250211}
\nolinenumbers

\usepackage{amsfonts}
\usepackage{amssymb}
\usepackage{tikz}
\usepackage{booktabs}
\usepackage{subfigure}
\allowdisplaybreaks
\usepackage{graphicx}
\usepackage{epstopdf}
\usepackage{algorithmic}

\ifpdf
  \DeclareGraphicsExtensions{.eps,.pdf,.png,.jpg}
\else
  \DeclareGraphicsExtensions{.eps}
\fi

\newsiamremark{remark}{Remark}
\newsiamremark{hypothesis}{Hypothesis}
\crefname{hypothesis}{Hypothesis}{Hypotheses}
\newsiamthm{claim}{Claim}
\newsiamremark{fact}{Fact}
\crefname{fact}{Fact}{Facts}

\headers{Space-time method for control problems with interfaces}{Q. H. Nguyen, P. C. Hoang, V. C. Le, and T. T. M. Ta}

\title{A space-time interface-fitted method for\\
moving-subdomain distributed control problems\\
with energy regularization\thanks{Submitted to the editors DATE.
\funding{This work was funded by the Vietnam Ministry of Education and Training under grant number B2024.BKA.18 and the European Research Council (ERC) under the European Union's Horizon 2020 research and innovation programme (grant agreement 101001847).}}}

\author{Quang Huy Nguyen\thanks{Faculty of Mathematics and Informatics, Hanoi University of Science and Technology, 11657 Hanoi, Vietnam 
  (\email{huy.nguyenquang1@hust.edu.vn, cuc.hp222163m@sis.hust.edu.vn, mai.tathithanh@hust.edu.vn}).}
\and Phuong Cuc Hoang\footnotemark[2]
\and Van Chien Le\thanks{IDLab, Department of Information Technology, Ghent University - imec, 9000 Ghent, Belgium (\email{vanchien.le@ugent.be}).}
\and Thi Thanh Mai Ta\footnotemark[2]}

\ifpdf
\hypersetup{
  pdftitle={Space-time method for control problems with interfaces},
  pdfauthor={Q. H. Nguyen, P. C. Hoang, V. C. Le, and T. T. M. Ta}
}
\fi

\input{lvc_style}

\begin{document}

\maketitle

% REQUIRED
\begin{abstract}
    This paper investigates a space-time interface-fitted approximation of a moving-interface optimal control problem with energy regularization. We reformulate the optimality conditions into a variational problem involving both the state and adjoint. This problem is shown to be equivalent to our optimal control problem. Based on fully unstructured, space-time interface-fitted meshes, we propose and analyze a Petrov-Galerkin approximation of the problem. An optimal error estimate with respect to a discrete norm is established under a specific regularity assumption on the state and adjoint. Several numerical results are presented to corroborate our theoretical results.
\end{abstract}

% REQUIRED

% REQUIRED

\section{Introduction}
\label{sec: introduction}

Let $\Omega$ be an open bounded domain in $\mathbb{R}^m\ \left(m=1, 2 \text{ or } 3\right)$ with Lipschitz boundary $\partial\Omega$. For each $t \in [0, T]$, with $T>0$, the domain $\Omega$ is partitioned into two subdomains, $\Omega_1(t)$ and $\Omega_2(t)$, separated by an interface $\Gamma(t)$. The interface $\Gamma(t)$ is transported by a velocity field $\vb \in \Cs\left(\left[0, T\right], \CCs^2\left(\overline{\Omega}\right)\right)$ satisfying $\nabla \cdot \vb\left(\xb, t\right) = 0$ for all $\left(\xb, t\right) \in \Omega \times \left[0, T\right]$.

We denote by $Q_T:= \Omega\times (0,T)$ the space-time cylinder, and by $Q_i\ \left(i=1,2\right)$ the two subdomains separated by the interface $\Gamma^\ast := \bigcup_{t\in (0,T)} \Gamma(t)\times \left\{t\right\}$. We assume that $\Gamma^\ast$ is a $\Cs^2$-regular hypersurface in $\mathbb{R}^{m+1}$ \cite[Remark~2.2]{VR2018}. Consider the following advection-diffusion problem
\begin{equation}
    \label{eq: state equation}
    \left\{\begin{array}{ll}
         \partial_t u + \vb \cdot \nabla u - \nabla \cdot\left(\kappa \nabla u\right) = f \qq & \text{in} \q Q_T, \\
         \left[u\right] = 0 & \text{on} \q \Gamma^\ast,\\
         \left[\kappa \nabla u \cdot \mathbf{n}\right] = 0 & \text{on} \q \Gamma^\ast,\\
         u = 0 & \text{on} \q \partial\Omega \times (0,T), \\
         u\left(\cdot, 0\right) = 0 & \text{in} \q \Omega,
    \end{array}\right.
\end{equation}
where $f$ denotes the source term and $\mathbf{n}$ is the unit normal on $\Gamma(t)$, pointing from $\Omega_1(t)$ into $\Omega_2(t)$. The notation $\left[u\right] := u_{1 \, \mid\, \Gamma(t)} - u_{2 \, \mid\, \Gamma(t)}$ represents the jump of $u$ across $\Gamma(t)$, with $u_{i \, \mid\,  \Gamma(t)}$ the limiting value of $u$ from $\Om_i(t)\ (i=1,2)$. For simplicity, we consider a subdomain-wise positive constant diffusion coefficient 
$$
\kappa := 
\begin{cases}
    \kappa_1 > 0 \qq & \text{in}\q Q_1, \\
    \kappa_2 > 0 & \text{in}\q Q_2.
\end{cases}
$$

In this work, we employ standard notations for Lebesgue spaces, isotropic and anisotropic Sobolev spaces, as well as their equipped norms \cite[Sections~1.4 and 2.2]{Ern2021}; see also \cite[Section~1.4.1]{WYW2006}. Particularly, the notation $\Hs^{1,0}\left(Q_T\right)$ denotes the anisotropic space
$$
\Hs^{1,0}\left(Q_T\right) := \left\{w\in \Ls^2\left(Q_T\right)\mid \partial_{\boldsymbol{x}}^{\boldsymbol{\alpha}} w \in \Ls^2\left(Q_T\right)\text{ for all } \abs{{\boldsymbol{\alpha}}}\le 1\right\},
$$
which is endowed with the norm
$$
\norm{w}^2_{\Hs^{1, 0}\left(Q_T\right)} := \sum_{\abs{{\boldsymbol{\alpha}}}\le 1} \norm{\partial_{\boldsymbol{x}}^{\boldsymbol{\alpha}} w}^2_{\Ls^2(Q_T)} \qqqqq \forall w \in \Hs^{1,0}\left(Q_T\right).
$$
Let $\Ws$ be the closure of $\Cs^1_0\left(Q_T\right)$ with respect to this norm. We equip the space $\Ws$ with the equivalent norm
$$
\norm{w}_{\Ws}^2 := \int\limits_{Q_T} \kappa \abs{\nabla w}^2 \dx \dt \qqqq \forall w\in \Ws.
$$
The notation $\Ws^\prime$ denotes the dual space of $\Ws$. Given a desired state $u_d\in \Ls^2\left(Q_T\right)$ and a regularization parameter $\eta > 0$, we are interested in the following optimal control problem
\begin{equation}
    \label{eq: original problem}
    \begin{aligned}
        & \min_{f\in \Ws^\prime} J\left(f\right):= \dfrac{1}{2}\norm{u\left(f\right) - u_d}^2_{\Ls^2\left(Q_T\right)} + \dfrac{\eta}{2}\norm{f}^2_{\Ws^\prime}, \\
        & \text{ subject to \eqref{eq: state equation}}.
    \end{aligned}
\end{equation}
Here, we use the notation $u\left(f\right)$ to emphasize the dependence of $u$ in \eqref{eq: state equation} on $f$. This model can arise in the context of heat transfer \cite{Slodicka2021}, the induction heating process \cite{LSV2024}, or the eddy-current problem in electromagnetics \cite{LSV2021a,LSV2021b}, where the goal is to find $f$ such that the corresponding $u\left(f\right)$ best approximates $u_d$. The choice of 
$f\in \Ws^\prime$ is motivated by the well-posedness of the variational formulation of Problem \eqref{eq: state equation}, as established in \cite[Theorem~2.1]{NLPT2024}. Compared to the analogous optimal control problem employing $\Ls^2\left(Q_T\right)$-regularization, Problem \eqref{eq: original problem} is more challenging due to the enlargement of the control space.

\subsection{Relevant literature}

Within the scope of finite element methods (FEMs), numerous techniques for discretizing space and time in time-dependent optimal control problems have been thoroughly investigated. These techniques are typically classified into two principal categories: separate discretization and "all-at-once" discretization. The former involves the combination of a temporal discretization scheme, such as the Euler method \cite{SS2020, GQ2023}, the Crank-Nicolson method \cite{MV2011, VHV2015, ZZH2023}, or the discontinuous Galerkin method \cite{BMV2007, MV2008, NV2012}, with a particular FEM for spatial discretization. Notably, combining adaptive techniques or parallel solvers with these methods can be challenging \cite{LMN2016, LMS2019}. In contrast, the second approach, known as the space-time method, treats the time variable as an additional spatial variable, with the temporal derivative of the solution considered as a convection term along the time direction \cite{LSTY2021, LS2022}. This method has been applied in \cite{LSTY2021b, LSY2024} to discretize parabolic optimal control problems with energy regularization. In \cite{LSTY2021b}, the authors established optimal error estimates for the state and adjoint in a discrete trial space norm. Using the perturbed Schur complement equation, they subsequently derived a higher-order estimate in the $\Ls^2\left(Q_T\right)$-norm for the state in \cite{LSY2024}. However, in all the aforementioned papers, the spatial domain does not contain an interface.

When the spatial domain is split by an interface, the solutions to PDE problems often exhibit low global spatial regularity, which prevents classical FEMs from converging at optimal rates \cite{Babuka1970}. Over the past few decades, numerous approaches have been developed to overcome this limitation. Prominent examples include the extended FEM \cite{Zunino2013}, the immersed FEM \cite{HLLY2013, Guo2021}, and the interface-fitted FEM \cite{CZ1998}. Despite these advances, time-dependent optimal control problems involving interfaces have received relatively little attention. To the best of our knowledge, only the papers \cite{ZLW2020, MHH2025} have addressed such problems, in which they employed the classical $\Ls^2\left(Q_T\right)$-regularization. In \cite{ZLW2020}, Zhang et al. employed the immersed FEM to discretize an optimal control problem for a parabolic transmission equation. They derived optimal error estimates for the control, state, and adjoint in various cases. In contrast, using the interface-fitted FEM, Ta et al. \cite{MHH2025} discretized an inverse problem for the advection-diffusion equation. Under suitable assumptions, the authors obtained different error estimates for three source discretizations. 

\subsection{Contributions and outline}

This paper proposes a numerical method for the moving-subdomain distributed control problem \eqref{eq: original problem}. Our approach is to eliminate the control, thereby reformulating the optimality conditions into a variational problem involving both the state and adjoint. The problem is then discretized in a space-time setting, where finite element spaces are constructed on fully unstructured, interface-fitted meshes \cite{NLPT2024}. In the context of moving-interface problems, this method eliminates the need for re-meshing at each time step while still providing a reasonable interface resolution. Moreover, it is flexible with respect to adaptability and parallelization. Unlike \cite{NPS2011, GHZ2012, KYK2014, LSTY2021b, FK2023}, where the authors discretized the optimality conditions via equivalent variational problems and established error estimates under global regularity assumptions of the state and adjoint, our work focuses on addressing the non-smoothness across the interface of these variables. Specifically, to interpolate functions with locally high but globally low regularity, we employ Stein extension operators \cite{Stein1971, Ern2021}. We derive a nearly optimal estimate for the one-dimensional spatial domain and a sub-optimal estimate for the two-dimensional spatial domain. Under stronger assumptions, the convergence rate is optimal for both one-dimensional and two-dimensional spatial domains.

The manuscript is structured as follows. In \Cref{sec: preliminaries}, we first present the essential functional settings, followed by a study of the existence and uniqueness of solutions to Problem \eqref{eq: original problem}, along with the corresponding optimality conditions. These optimality conditions are then reformulated into Problem \eqref{eq: coupled state-adjoint}, which is a variational problem involving both the state and adjoint. We discuss the unique solvability of this problem in \Cref{sec: reduced optimality system}. Next, \Cref{sec: finite element discretization} presents the finite element discretization of Problem \eqref{eq: coupled state-adjoint}, and different error estimates for the state and adjoint are established under specific regularity assumptions. Finally, we illustrate our theoretical results with numerical examples in \Cref{sec: numerical results}.

\section{Preliminaries}
\label{sec: preliminaries}

For $w\in \Ws$, its distributional time derivative is defined as
$$
\partial_t w\left(\xi\right) := - \int\limits_{Q_T} w \partial_t \xi \dx \dt\qqqq \fa \xi \in \Cs_0^1\left(Q_T\right).
$$
Then, we introduce the space $\Us := \left\{w\in \Ws\mid \partial_t w \in \Ws^\prime \right\}$ and its two subspaces
$$
\Us_0 := \left\{w\in \Us\mid w\left(\cdot, 0\right) = 0\right\} \qqq \text{and} \qqq \Us_T := \left\{w\in \Us\mid w\left(\cdot, T\right) = 0\right\},
$$
furnished with the graph norm
$$
\norm{w}^2_{\Us} := \norm{w}^2_{\Ws} + \norm{\partial_t w}^2_{\Ws^\prime}\qqqq \forall w\in \Us.
$$
It is worth mentioning that the space $\Us$ is identical to the Bochner space 
$$
\brac{w \in \Ls^2\left(\left(0, T\right), \Hs^1_0\left(\Om\right)\right) \mid \pa_t w \in \Ls^2\left(\left(0, T\right), \Hs^{-1}\left(\Om\right)\right)},
$$
allowing us to define the traces at $t = 0$ and $t = T$ in $\Ls^2\left(\Om\right)$ \cite[Lemma~7.3]{Roubicek2005}. Therefore, the spaces $\Us_0$ and $\Us_T$ are well-defined. 

The variational formulation of Problem \eqref{eq: state equation} is stated as follows: Given $f\in \Ws^\prime$, determine $u\in \Us_0$ such that
\begin{equation}
    \label{eq: weak state equation}
    a\left(u, \vphi\right) = \inprod{f, \vphi} \qqqq \forall\vphi \in \Ws,
\end{equation}
where the bilinear form $a : \Us_0\times \Ws \rightarrow \R$ is given by
$$
a\left(u, \vphi\right) := \inprod{\partial_t u, \vphi} +\int\limits_{Q_T}  \left(\vb\cdot\nabla u\right)\vphi + \kappa \nabla u \cdot \nabla \vphi \dx \dt,
$$
and $\inprod{\cdot, \cdot}$ denotes the duality pairing between $\Ws^\prime$ and $\Ws$. This problem is well-posed \cite[Theorem~2.1]{NLPT2024}. Therefore, Problem \eqref{eq: original problem} is reduced to the following form
\begin{equation}
    \label{eq: problem formulation}
    \begin{aligned}
        & \min_{f\in \Ws^\prime} J\left(f\right), \\
        & \text{ subject to \eqref{eq: weak state equation}}.
    \end{aligned}
\end{equation}
Since $J$ is a continuous and strictly convex functional on the Hilbert space $\Ws^\prime$, this problem admits a unique solution $f^\ast\in \Ws^\prime$, see the proof of \cite[Theorem~2.17]{Troltzsch2010}. 

Next, we derive the optimality conditions. To do so, we introduce the following adjoint problem: Find $p\left(f\right)\in \Us_T$ such that
\begin{equation}
    \label{eq: weak adjoint equation}
    a^{\prime}\left(p\left(f\right),\phi\right) = \int\limits_{Q_T}\left(u\left(f\right)-u_d\right)\phi \dx\dt \qqqq \forall \phi\in \Ws,
\end{equation}
where the bilinear form $a^\prime: \Us_T \times \Ws \rightarrow \mathbb{R}$ is defined as
$$
a^{\prime}\left(p,\phi\right) := - \inprod{\partial_t p,\phi} +\int\limits_{Q_T} -\left(\vb\cdot\nabla p\right)\phi + \kappa \nabla p \cdot \nabla \phi \dx \dt.
$$
After changing the directions of time and the velocity field, we apply \cite[Lemmas~2.1 and 2.2]{NLPT2024} to conclude the well-posedness of this problem. 

We also need the Riesz representation: For any $\widetilde{w}\in \Ws^\prime$, there exists a unique $z\left(\widetilde{w}\right)\in \Ws$ such that
\begin{equation}
    \label{eq: riesz representation 1}
    \int\limits_{Q_T} \kappa \nabla z\left(\widetilde{w}\right) \cdot \nabla \zeta \dx \dt = \inprod{\widetilde{w}, \zeta} \qqqq \forall \zeta \in \Ws.
\end{equation}
Moreover, we have
\begin{equation}
    \label{eq: riesz representation 2}
    \norm{\widetilde{w}}_{\Ws^\prime}^2 = \inprod{\widetilde{w}, z\left(\widetilde{w}\right)} = \norm{z\left(\widetilde{w}\right)}_{\Ws}^2.
\end{equation}

\begin{theorem}
    \label{theo: optimality conditions}
    An element $f^\ast  \in \Ws^\prime$ is the solution to Problem \eqref{eq: problem formulation} if and only if there exist $u^\ast  \in \Us_0$ and  $p^\ast  \in \Us_T$ such that
    \begin{equation}
        \label{eq: optimality conditions 1}
        a\left(u^\ast , \vphi\right) = \inprod{f^\ast , \vphi}\qqqq \forall\vphi \in \Ws,
    \end{equation}
    and
    \begin{equation}
        \label{eq: optimality conditions 2}
        a^{\prime}\left(p^\ast ,\phi\right) = \int\limits_{Q_T}\left(u^\ast -u_d\right)\phi \dx\dt \qqqq \forall \phi\in \Ws,
    \end{equation}
    along with the gradient condition
    \begin{equation}
        \label{eq: optimality conditions 3}
        p^\ast  + \eta z\left(f^\ast\right) = 0_{\Ws}.
    \end{equation}
\end{theorem}

\begin{proof}
    Take a small variation $\delta f \in \Ws^\prime$ of $f\in \Ws^\prime$. We observe that
    \begin{align*}
        & \norm{u\left(f + \delta f \right)-u_d}^2_{\Ls^2\left(Q_T\right)} - \norm{u\left(f\right) -u_d}^2_{\Ls^2\left(Q_T\right)} =\\
        & \q =\norm{u\left(f + \delta f \right)}^2_{\Ls^2\left(Q_T\right)} - \norm{u\left(f \right)}^2_{\Ls^2\left(Q_T\right)} - 2\int\limits_{Q_T}\left[u\left(f + \delta f \right) - u\left(f\right)\right]u_d\dx\dt\\
        & \q =\norm{u\left(\delta f\right)}^2_{\Ls^2\left(Q_T\right)} + 2\int\limits_{Q_T}u\left(\delta f\right)\left(u\left(f\right)-u_d\right)\dx\dt.
    \end{align*}
    Moreover, by using \eqref{eq: riesz representation 1} and \eqref{eq: riesz representation 2}, we have
    $$
    \norm{f+\delta f}_{\Ws^\prime}^2 - \norm{f}_{\Ws^\prime}^2= \norm{z\left(f+ \delta f\right)}^2_{\Ws} - \norm{z\left(f\right)}^2_{\Ws}= \norm{\delta f}^2_{\Ws^\prime} + 2\inprod{f, z\left(\delta f\right)}.
    $$
    Since $\norm{u\left(\delta f\right)}^2_{\Ls^2\left(Q_T\right)} = o\left(\norm{\delta f}_{\Ws^\prime}\right)$ as $\delta f \rightarrow 0_{\Ws^\prime}$, we obtain
    \begin{equation}
        \label{eq: optimality conditions 4}
        J\left(f + \delta f \right) - J\left(f \right) = o\left(\norm{\delta f}_{\Ws^\prime}\right) + \int\limits_{Q_T}u\left(\delta f\right)\left(u\left(f\right)-u_d\right)\dx\dt + \eta \inprod{f, z\left(\delta f\right)}.
    \end{equation}
    Now, we choose $\phi = u\left(\delta f\right)\in \Us_0$ in \eqref{eq: weak adjoint equation} to get 
    \begin{align*}
        & \int\limits_{Q_T}u\left(\delta f\right)\left(u\left(f\right)-u_d\right)\dx\dt =  \\
        & \q = - \inprod{\partial_t \left(p\left(f\right)\right), u\left(\delta f\right)} +\int\limits_{Q_T}  -\left[\vb\cdot\nabla p\left(f\right)\right]u\left(\delta f\right) + \kappa \nabla p\left(f\right) \cdot \nabla u\left(\delta f\right) \dx \dt.
    \end{align*}
    On the one hand, we integrate by parts with $p\left(f\right)\in \Us_T$ and $u\left(\delta f\right) \in \Us_0$ to obtain 
    $$
    \inprod{\partial_t \left(p\left(f\right)\right), u\left(\delta f\right)} = - \inprod{\partial_t \left(u\left(\delta f\right)\right), p\left(f\right)}.
    $$
    On the other hand, thanks to the divergence theorem, the homogeneous Dirichlet boundary condition, and the assumption $\nabla\cdot \vb \left(\xb, t\right)= 0$ for all $\left(\xb, t\right)\in \Omega\times [0, T]$, we deduce that
    \begin{equation}
        \label{eq: optimality conditions 5}
        \begin{aligned}
            &\int\limits_{Q_T} \left[\vb\cdot\nabla p\left(f\right)\right] u\left(\delta f\right) \dx\dt =\\
            & \q = \int\limits_0^T\int\limits_{\partial\Omega} p\left(f\right)u\left(\delta f\right) \vb \cdot \mathbf{n}^\prime \ds\dt- \int\limits_{Q_T} \left[\vb\cdot\nabla u\left(\delta f\right)\right]p\left(f\right) + p\left(f\right)u\left(\delta f\right)\left(\nabla\cdot \vb\right)\dx\dt \\
            & \q = - \int\limits_{Q_T} \left[\vb\cdot\nabla u\left(\delta f\right)\right]p\left(f\right) \dx\dt,
        \end{aligned}
    \end{equation}
    where $\mathbf{n}^\prime$ denotes the outward normal to $\partial\Omega$. 
    Hence, we obtain
    $$
    \int\limits_{Q_T}u\left(\delta f\right)\left(u\left(f\right)-u_d\right)\dx\dt = a\left(u\left(\delta f\right), p\left(f\right)\right) = \inprod{\delta f, p\left(f\right)}.
    $$
    By substituting this into \eqref{eq: optimality conditions 4} and applying \eqref{eq: riesz representation 1}, we get
    $$
    J\left(f + \delta f \right) - J\left(f \right) = o\left(\norm{\delta f}_{\Ws^\prime}\right) + \inprod{\delta f, p\left(f\right)} + \eta\inprod{\delta f, z\left(f\right)}.
    $$
    Therefore, the gradient of $J$ at $f\in \Ws^\prime$ is $\nabla J\left(f\right) = p\left(f\right) + \eta z\left(f\right)\in \Ws$. We finish the proof by noting that $\nabla J\left(f^\ast\right) = 0_{\Ws}$.
\end{proof}

\section{Reduced optimality system}
\label{sec: reduced optimality system}

Following the approach of Langer et al. in \cite[Section~3]{LSTY2021}, we now reduce the system \eqref{eq: optimality conditions 1}--\eqref{eq: optimality conditions 3} into a variational problem where the control $f^\ast$ is absent. We then prove that solving that problem is equivalent to solving Problem \eqref{eq: problem formulation}, and hence implicitly obtain the unique solvability of the problem.

To eliminate $f^\ast$ from the system \eqref{eq: optimality conditions 1}--\eqref{eq: optimality conditions 3}, we use \eqref{eq: optimality conditions 1}, \eqref{eq: riesz representation 1}, and \eqref{eq: optimality conditions 3}. We arrive at
$$
a\left(u^\ast , \vphi\right) = \inprod{f^\ast , \vphi} = \int\limits_{Q_T} \kappa \nabla z\left(f^\ast\right)  \cdot \nabla \vphi \dx \dt = -\dfrac{1}{\eta} \int\limits_{Q_T} \kappa \nabla p^\ast  \cdot \nabla \vphi \dx \dt \qq \forall \vphi\in \Ws,
$$
and hence
\begin{equation}
    \label{eq: coupled state-adjoint eq1}
    a\left(u^\ast , \vphi\right) + \dfrac{1}{\eta} \int\limits_{Q_T} \kappa \nabla p^\ast  \cdot \nabla \vphi \dx \dt = 0 \qqqq \forall \vphi\in \Ws.
\end{equation}
Notably, this equation expresses a relation between $u^\ast$ and $p^\ast$ without the explicit presence of $f^\ast$. We further manipulate \eqref{eq: optimality conditions 2}. Now, by integrating by parts, we obtain
$$
-\inprod{\partial_t p^\ast, \phi} = \inprod{\partial_t \phi, p^\ast} \qqqq \forall \phi \in \Us_0.
$$
In addition, the arguments of \eqref{eq: optimality conditions 5} gives us
$$
-\int\limits_{Q_T} \left(\vb\cdot\nabla p^\ast\right) \phi \dx\dt = \int\limits_{Q_T} \left(\vb\cdot\nabla \phi\right) p^\ast \dx\dt \qqqq \forall \phi \in \Ws.
$$
Therefore, we deduce that
\begin{equation}
    \label{eq: relation between adjoint and state}
    a^\prime\left(p^\ast, \phi\right) = a\left(\phi, p^\ast\right) \qqqq \forall \phi \in \Us_0,
\end{equation}
which turns \eqref{eq: optimality conditions 2} into
$$
\int\limits_{Q_T}u^\ast\phi\dx\dt - a\left(\phi, p^\ast\right) = \int\limits_{Q_T}u_d\phi\dx\dt\qqqq \forall \phi\in \Us_0.
$$
Motivated by \eqref{eq: coupled state-adjoint eq1} and this equation, our idea is to solve Problem \eqref{eq: problem formulation} via the following variational problem: Given $u_d\in \Ls^2\left(Q_T\right)$, find $\left(\overline{u}, \overline{p}\right)\in \Us_0\times \Ws$ such that
\begin{equation}
    \label{eq: coupled state-adjoint}
    \mathcal{A}\left(\left(\overline{u}, \overline{p}\right), \left(\vphi, \phi\right)\right) = \int\limits_{Q_T}u_d\phi\dx\dt\qqqq \forall \left(\vphi, \phi\right) \in \Ws\times \Us_0,
\end{equation}
where the bilinear form $\mathcal{A}: \left(\Us_0\times \Ws\right)\times \left(\Ws\times \Us_0\right)\rightarrow \mathbb{R}$ is given by
$$
\mathcal{A}\left(\left(u, p\right), \left(\vphi, \phi\right)\right) := a\left(u, \vphi\right) + \dfrac{1}{\eta}\int\limits_{Q_T}\kappa \nabla p \cdot\nabla \vphi\dx\dt + \int\limits_{Q_T}u\phi\dx\dt - a\left(\phi, p\right).
$$
Compared with the problems \eqref{eq: optimality conditions 1}--\eqref{eq: optimality conditions 2}, the trial and test spaces for the state remain unchanged. However, for the adjoint, the trial space is now bigger, while the test space is smaller. In other words, Problem \eqref{eq: coupled state-adjoint} is obtained by enlarging the trial spaces and restricting the test spaces in the original problems. This property, combined with the solvability of the problems \eqref{eq: optimality conditions 1}--\eqref{eq: optimality conditions 2}, guarantees the existence of at least one solution to Problem \eqref{eq: coupled state-adjoint}.

Before concluding the unique solvability of this problem, let us first prove its equivalence with the system \eqref{eq: optimality conditions 1}--\eqref{eq: optimality conditions 3}. Specifically, we have the following result:

\begin{theorem}
    \label{theo: problems equivalence}
    Let $\left(\overline{u}, \overline{p}\right)\in \Us_0\times \Ws$ be a solution to Problem \eqref{eq: coupled state-adjoint}. Then, we have $\left(\overline{u}, \overline{p}\right)\in \Us_0\times \Us_T$. Moreover, $\overline{u} = u^\ast$ in $\Us_0$ and $\overline{p} = p^\ast$ in $\Us_T$, where $u^\ast \in \Us_0$ and $p^\ast\in \Us_T$ are introduced in \cref{theo: optimality conditions}.
\end{theorem}

\begin{proof}
    The first statement can be proved by showing that $\overline{p}\in \Us_T$. To do so, we choose $\vphi = 0_{\Ws}$ in \eqref{eq: coupled state-adjoint} to get
    \begin{equation}
        \label{eq: problems equivalence 1}
        \inprod{\partial_t \phi, \overline{p}} = - \int\limits_{Q_T}  \left(\vb\cdot\nabla \phi \right)\overline{p} + \kappa \nabla \phi \cdot \nabla \overline{p} \dx \dt + \int\limits_{Q_T}\left(\overline{u} - u_d\right)\phi\dx\dt \qq \forall \phi \in \Us_0.
    \end{equation}
    For any $\phi\in \Cs^1_0\left(Q_T\right)\subset \Us_0$ and $\overline{p}\in \Ws$, we clearly have (by the definition of $\pa_t \ovl{p}$)
    $$
    \inprod{\partial_t \phi, \overline{p}} = \int\limits_{Q_T} \left(\partial_t \phi \right)\overline{p}\dx\dt = - \partial_t \overline{p}\left(\phi\right),
    $$
    which makes \eqref{eq: problems equivalence 1} become
    $$
    - \partial_t \overline{p}\left(\phi\right) = - \int\limits_{Q_T}  \left(\vb\cdot\nabla \phi \right)\overline{p} + \kappa \nabla \phi \cdot \nabla \overline{p} \dx \dt + \int\limits_{Q_T}\left(\overline{u} - u_d\right)\phi\dx\dt \qq \forall \phi \in \Cs^1_0\left(Q_T\right).
    $$
    Since $\Cs^1_0\left(Q_T\right)$ is dense in $\Ws$, the right-hand side of this equality defines a bounded linear functional of $\phi$ on $\Ws$. Therefore, we imply that $\partial_t \overline{p}\in \Ws^\prime$, and hence $\overline{p}\in \Us$. Moreover, the following equality holds
    $$
    -\inprod{\partial_t \overline{p}, \phi} = -\int\limits_{Q_T}  \left(\vb\cdot\nabla \phi \right)\overline{p} + \kappa \nabla \phi \cdot \nabla \overline{p} \dx \dt + \int\limits_{Q_T}\left(\overline{u} - u_d\right)\phi\dx\dt \qq \forall \phi \in \Ws.
    $$
    We then take $\phi \in \Us_0 \sst \Ws$ and subtract this from \eqref{eq: problems equivalence 1} to obtain $\inprod{\partial_t \phi, \overline{p}} + \inprod{\partial_t \overline{p}, \phi} = 0$, which leads to
    $$
    \int\limits_{\Omega} \phi\left(\xb, T\right)\overline{p}\left(\xb, T\right)\dx = 0 \qqqq \forall \phi \in \Us_0,
    $$
    using the integration by parts. Therefore, we conclude that $\overline{p}\left(\cdot, T\right) = 0$, and thus $\overline{p} \in \Us_T$.

    We now prove the second statement. Thanks to the arguments for \eqref{eq: relation between adjoint and state}, we can rewrite \eqref{eq: problems equivalence 1} as
    $$
    a^\prime\left(\overline{p}, \phi\right) = \int\limits_{Q_T}\left(\overline{u} - u_d\right)\phi\dx\dt \qqqq \forall \phi \in \Us_0.
    $$
    By invoking the density of $\Us_0$ in $\Ws$, we hence imply 
    \begin{equation}
        \label{eq: problems equivalence 2}
        a^\prime\left(\overline{p}, \phi\right) = \int\limits_{Q_T}\left(\overline{u} - u_d\right)\phi\dx\dt \qqqq \forall \phi \in \Ws,
    \end{equation}
    which recovers Problem \eqref{eq: optimality conditions 2}. Subsequently, we choose $\phi = 0_{\Us_0}$ in \eqref{eq: coupled state-adjoint} to obtain
    $$
    a\left(\overline{u}, \vphi\right) = -\dfrac{1}{\eta} \int\limits_{Q_T}  \kappa \nabla \overline{p}\cdot \nabla \vphi \dx \dt \qqqq \forall \vphi \in \Ws.
    $$
    By further defining $\overline{z}:= -\frac{1}{\eta}\overline{p}$ in $\Ws$ and using \eqref{eq: riesz representation 1}, we conclude the existence of an element $\overline{f}\in \Ws^\prime$ such that
    \begin{equation}
        \label{eq: problems equivalence 3}
        a\left(\overline{u}, \vphi\right) = \inprod{\overline{f}, \vphi} \qqqq \forall \vphi \in \Ws,
    \end{equation}
    which corresponds precisely to Problem \eqref{eq: optimality conditions 1}. The conclusion follows by combining \Cref{theo: optimality conditions} with \eqref{eq: problems equivalence 2} and \eqref{eq: problems equivalence 3}.
\end{proof}

From this theorem, we see that the system \eqref{eq: optimality conditions 1}--\eqref{eq: optimality conditions 3} is equivalent to Problem \eqref{eq: coupled state-adjoint}. Notably, the latter offers the advantage of enabling an all-at-once space-time treatment of the state and adjoint, as opposed to classical gradient-based sequential methods. As a result, we consider discretizing Problem \eqref{eq: coupled state-adjoint} as an efficient way to discretize our optimal control problem.

Moreover, \Cref{theo: problems equivalence} guarantees the uniqueness of the solution to Problem \eqref{eq: coupled state-adjoint}. This follows from the unique determination of the pair $\left(u^\ast, p^\ast\right)\in \Us_0\times \Us_T$ and the second part of the theorem. We thus obtain the following result:

\begin{corollary}
    \label{cor: uniquess of solution to coupled state-adjoint problem}
    Given $u_d\in \Ls^2\left(Q_T\right)$. Then, Problem \eqref{eq: coupled state-adjoint} admits a unique solution $\left(\overline{u}, \overline{p}\right)\in \Us_0\times \Us_T$.
\end{corollary}

Finally, it is important to note that once Problem \eqref{eq: coupled state-adjoint} is solved, the optimal control $f^\ast$ can be obtained directly via Problem \eqref{eq: optimality conditions 1}.

\section{Finite element discretization}
\label{sec: finite element discretization}

Assume that $\Omega$ is a polyhedron in $\mathbb{R}^m$ $\left(m=1\text{ or } 2\right)$. Let $\left\{\mathcal{T}_h\right\}_{h\in \left(0,h^\dagger\right)}$ be a family of quasi-uniform, interface-fitted meshes of the cylinder $Q_T = \Omega\times \left(0,T\right)$, with mesh size $h\in \left(0,h^\dagger\right)$, where $h^\dagger>0$ is given \cite[Section~3.1]{NLPT2024}. Notably, this triangulation does not allow $\Gamma^\ast$ to cut the triangles (or tetrahedra) arbitrarily. In particular, $\Gamma^\ast$ goes through exactly two (or three) vertices of all interface triangles (or tetrahedra), and forms with each corresponding edge (or face) a small discrepancy region. We denote by $\Gamma_h^\ast$ the linear approximation of $\Gamma^\ast$, consisting of all edges (or faces) with vertices lying on $\Gamma^\ast$. This discrete interface partitions $Q_T$ into two subdomains $Q_{1,h}$ and $Q_{2,h}$. Finally, the mismatch between space-time subdomains and their discrete counterparts is defined as
$$
S_h:= \left(Q_1\setminus \overline{Q_{1,h}}\right) \cup \left(Q_2\setminus \overline{Q_{2,h}}\right).
$$

Throughout this section, the notation $C$ denotes a generic positive constant that
does not depend on the control $f^\ast$ or the mesh size $h$, but may depend on the cylinder $Q_T$, the position of $\Gamma^\ast$, the norm $\norm{\vb}_{\LLs^\infty\left(Q_T\right)}$, the coefficient $\kappa$, and the parameter $\eta$. Its value may vary depending on the context.

\subsection{Space-time interface-fitted method}

On the mesh $\mathcal{T}_h$, we define the finite element spaces as
$$
\Ws_h :=\left\{w_h \in \Cs\left(\overline{Q_T}\right)\mid w_{h\, \mid \, K}\in \mathbb{P}_1\left(K\right)\ \forall K\in \mathcal{T}_h\right\}\cap \Ws,
$$
and
$$
\Us_h :=\Ws_h \cap \Us_0.
$$
The Petrov-Galerkin approximation of Problem \eqref{eq: coupled state-adjoint} is then given by: Determine $\left(\overline{u}_h, \overline{p}_h\right)\in \Us_h\times \Ws_h$ such that
\begin{equation}
    \label{eq: discrete coupled state-adjoint}
    \mathcal{A}_h\left(\left(\overline{u}_h, \overline{p}_h\right), \left(\vphi_h, \phi_h\right)\right) = \int\limits_{Q_T}u_d\phi_h\dx\dt \qqqq \forall \left(\vphi_h, \phi_h\right) \in \Ws_h\times \Us_h,
\end{equation}
where the bilinear form $\mathcal{A}_h: \left(\Us_0\times \Ws\right)\times \left(\Ws\times \Us_0\right)\rightarrow \mathbb{R}$ is given by
$$
\mathcal{A}_h\left(\left(u, p\right), \left(\vphi, \phi\right)\right) := a_h\left(u, \vphi\right) + \dfrac{1}{\eta}\int\limits_{Q_T}\kappa_h \nabla p \cdot\nabla \vphi\dx\dt + \int\limits_{Q_T}u\phi\dx\dt - a_h\left(\phi, p\right).
$$
The bilinear form $a_h:\Us_0 \times \Ws\rightarrow \mathbb{R}$ is defined as
$$
a_h\left(u, \vphi\right) := \inprod{\partial_t u,\vphi} + \int\limits_{Q_T} \left(\vb\cdot \nabla u\right)\vphi +  \kappa_h \nabla u \cdot \nabla \vphi \dx \dt,
$$
and $\kappa_h$ is a mesh-dependent positive constant
$$
\kappa_h := 
\begin{cases}
    \kappa_{1}>0 & \text{in} \q Q_{1,h}, \\ 
    \kappa_{2}>0 & \text{in} \q Q_{2,h}.
\end{cases}
$$

Regarding numerical analysis, let us introduce the following seminorms
$$
\vertiii{w}^2 : = \sum_{i=1}^2\int\limits_{Q_{i,h}} \kappa_i \abs{\nabla w}^2 \dx \dt \qqqq \fa w \in \Hs^{1,0}\left(Q_{1,h}\cup Q_{2,h}\right),
$$
and
$$
\vertiii{w}_{\ast}^2 := \vertiii{w}^2 + \vertiii{z_h\left(\partial_t w\right)}^2\qqqq \forall w \in \Hs^{1}\left(Q_{1,h}\cup Q_{2,h}\right),
$$
where $z_h\left(\partial_t w\right)\in \Ws_h$ denotes the unique solution to the problem
\begin{equation}
    \label{eq: discrete riesz representation}
    \int\limits_{Q_T} \kappa_h \nabla z_h\left(\partial_t w\right) \cdot \nabla \zeta_h \dx \dt = \sum_{i=1}^2\int\limits_{Q_{i,h}} \left(\partial_t w\right)\zeta_h \dx\dt \qqqq \forall \zeta_h\in \Ws_h.
\end{equation}
Since these seminorms are defined for functions whose gradient may be discontinuous across $\Gamma_h^\ast$, they are well-suited for the analysis of discrete problems. In what follows, we equip the spaces $\Ws_h$ and $\Us_h$ with the seminorms $\vertiii{\cdot}$ and $\vertiii{\cdot}_\ast$, respectively. 

\begin{lemma}
    \label{lem: discrete bnb1}
    There exists a constant $C>0$ such that for all $\left(u_h, p_h\right)\in \Us_h\times \Ws_h$, the following stability condition holds
    $$
    \sup_{\left(\vphi_h, \phi_h\right) \in \Ws_h\times \Us_h\setminus \left\{\left(0_{\Ws_h},0_{\Us_h}\right)\right\}} \dfrac{\mathcal{A}_h\left(\left(u_h, p_h\right), \left(\vphi_h, \phi_h\right)\right)}{\sqrt{\vertiii{\vphi_h}^2 + \vertiii{\phi_h}^2_\ast}} \ge C \sqrt{\vertiii{u_h}^2_\ast + \vertiii{p_h}^2}.
    $$
\end{lemma}

\begin{proof}
    We first fix $\left(u_h,p_h\right)\in \left(\Us_h, \Ws_h\right)$ and choose $\vphi_h = \mu u_h + z_h\left(\partial_t u_h\right)+\lambda p_h\in \Ws_h$ and $\phi_h = \lambda u_h\in \Us_h$, where $\mu>0, \lambda>0$ are sufficiently large numbers that will be specified later. Then, we have
    \begin{equation}
        \label{eq: bnb1 1}
        \begin{aligned}
            &\mathcal{A}_h\left(\left(u_h, p_h\right), \left(\vphi_h, \phi_h\right)\right) = \\
            & \, = \mu a_h\left(u_h,u_h\right) + a_h\left(u_h,z_h\left(\partial_t u_h\right)\right) + \dfrac{1}{\eta}\int\limits_{Q_T} \kappa_h \nabla p_h \cdot \nabla \left(\mu u_h + z_h\left(\partial_t u_h\right)\right)\dx\dt \\
            & \q + \dfrac{\lambda}{\eta}\vertiii{p_h}^2 + \lambda \norm{u_h}^2_{\Ls^2\left(Q_T\right)}.
        \end{aligned}
    \end{equation}
    For the first term on the right-hand side of \eqref{eq: bnb1 1}, we integrate by parts with $u_h\in \Us_h$ and apply the technique in \eqref{eq: optimality conditions 5}. The following estimate holds
    \begin{equation}
        \label{eq: bnb1 2}
        a_h\left(u_h,u_h\right) = \dfrac{1}{2}\norm{u_h\left(\cdot, T\right)}^2_{\Ls^2\left(\Omega\right)} - \dfrac{1}{2}\norm{u_h\left(\cdot, 0\right)}^2_{\Ls^2\left(\Omega\right)} + \vertiii{u_h}^2\ge \vertiii{u_h}^2.
    \end{equation}
    By using \eqref{eq: discrete riesz representation} and the Cauchy-Schwarz inequality, we bound the second term on the right-hand side of \eqref{eq: bnb1 1} as follows
    \begin{align*}
        & a_h\left(u_h, z_h\left(\partial_t u_h\right)\right)= \\
        & \q = \vertiii{z_h\left(\partial_t u_h\right)}^2 + \int\limits_{Q_T}  \left(\vb\cdot\nabla u_h \right)z_h\left(\partial_t u_h\right) + \kappa_h \nabla u_h \cdot \nabla z_h\left(\partial_t u_h\right) \dx \dt \\
        & \q \ge \vertiii{z_h\left(\partial_t u_h\right)}^2 - \left(\dfrac{C}{4\epsilon}\vertiii{u_h}^2 + \epsilon\norm{z_h\left(\partial_t u_h\right)}^2_{\Ls^2\left(Q_T\right)}\right)\\
        & \qq - \left(\dfrac{1}{4\epsilon}\vertiii{u_h}^2 + \epsilon\vertiii{z_h\left(\partial_t u_h\right)}^2\right),
    \end{align*}
    for any $\epsilon>0$. Moreover, we recall from \cite[Lemma~3.27]{Ern2021} the Poincar{\' e}–Steklov inequality
    \begin{equation}
        \label{eq: compare L2 and the W norm}
        \norm{w}_{\Ls^2\left(Q_T\right)}\le C\vertiii{w}\qqqq \forall w\in \Ws.
    \end{equation}
    After applying this inequality and rearranging the terms, we end up with
    \begin{equation}
        \label{eq: bnb1 3}
        a_h\left(u_h, z_h\left(\partial_t u_h\right)\right) \ge \left(1 - C\epsilon\right)\vertiii{z_h\left(\partial_t u_h\right)}^2 - \dfrac{C}{\epsilon}\vertiii{u_h}^2.
    \end{equation}
    We now estimate the third term on the right-hand side of \eqref{eq: bnb1 1}. By the Cauchy-Schwarz inequality, we deduce that
    \begin{align*}
        & \int\limits_{Q_T} \kappa_h \nabla p_h \cdot \nabla \left(\mu u_h + z_h\left(\partial_t u_h\right)\right)\dx\dt \ge \\
        & \q \ge -\mu\left(\dfrac{1}{2\eta}\vertiii{p_h}^2 + \dfrac{\eta}{2}\vertiii{u_h}^2\right) - \left(\dfrac{1}{4\epsilon}\vertiii{p_h}^2 + \epsilon\vertiii{z_h\left(\partial_t u_h\right)}^2\right)\\
        & \q = -\dfrac{\mu \eta}{2}\vertiii{u_h}^2 -\epsilon \vertiii{z_h\left(\partial_t u_h\right)}^2 -\left(\dfrac{\mu}{2\eta} + \dfrac{1}{4\epsilon}\right)\vertiii{p_h}^2.
    \end{align*}
    By inserting \eqref{eq: bnb1 2}, \eqref{eq: bnb1 3}, and this into \eqref{eq: bnb1 1}, we arrive at
    \begin{align*}
        \mathcal{A}_h\left(\left(u_h, p_h\right), \left(\vphi_h, \phi_h\right)\right) &\ge \left(\dfrac{\mu}{2} - \dfrac{C}{\epsilon}\right)\vertiii{u_h}^2 + \left(1-C\epsilon -\dfrac{\epsilon}{\eta}\right) \vertiii{z_h\left(\partial_t u_h\right)}^2 \\
        & \q + \dfrac{1}{\eta}\left(\lambda -\dfrac{\mu}{2\eta} - \dfrac{1}{4\epsilon}\right)\vertiii{p_h}^2.
    \end{align*}
    Here, we choose sufficiently small $\epsilon>0$ such that $1-C\epsilon -\frac{\epsilon}{\eta}\ge \frac{1}{2}$, then choose sufficiently large $\mu>0$ and $\lambda>0$ such that $\frac{\mu}{2} - \frac{C}{\epsilon}\ge \frac{1}{2}$ and $\lambda -\frac{\mu}{2\eta} - \frac{1}{4\epsilon} \ge \frac{\eta}{2}$. Consequently, we have
    \begin{align*}
        \mathcal{A}_h\left(\left(u_h, p_h\right), \left(\vphi_h, \phi_h\right)\right) &\ge \dfrac{1}{2}\left(\vertiii{u_h}^2 +\vertiii{z_h\left(\partial_t u_h\right)}^2 + \vertiii{p_h}^2\right) \\
        & = \dfrac{1}{2} \left(\vertiii{u_h}^2_\ast + \vertiii{p_h}^2\right).
    \end{align*}
    Moreover, it follows from the Cauchy-Schwarz inequality that
    \begin{align*}
        \vertiii{\vphi_h}^2 + \vertiii{\phi_h}^2_\ast &= \vertiii{\mu u_h + z_h\left(\partial_t u_h\right)+\lambda p_h}^2 + \vertiii{\lambda u_h}^2_\ast\\
        & \le \left(\mu^2 +1 +\lambda^2\right)\left(\vertiii{u_h}^2 + \vertiii{z_h\left(\partial_t u_h\right)}^2 + \vertiii{p_h}^2\right) + \lambda^2\vertiii{u_h}^2_\ast \\
        & \le C\left(\vertiii{u_h}^2_\ast + \vertiii{p_h}^2\right).
    \end{align*}
    Therefore, we conclude that
    $$
    \mathcal{A}_h\left(\left(u_h, p_h\right), \left(\vphi_h, \phi_h\right)\right) \ge C\sqrt{\vertiii{u_h}^2_\ast + \vertiii{p_h}^2} \sqrt{\vertiii{\vphi_h}^2 + \vertiii{\phi_h}^2_\ast}.
    $$
    By letting $\left(u_h,p_h\right)$ vary over $\Us_h\times \Ws_h$, we obtain the desired result.
\end{proof}

This lemma and the fact that the discrete trial and test spaces are of the same dimension ensure the unique solvability of Problem \eqref{eq: discrete coupled state-adjoint}. 

\begin{theorem}
    \label{theo: discrete well-posedness}
    For a given $u_d\in \Ls^2\left(Q_T\right)$, there exists a unique solution $\left(\overline{u}_h, \overline{p}_h\right)\in \Us_h\times \Ws_h$ to Problem \eqref{eq: discrete coupled state-adjoint}.
\end{theorem}

\subsection{Error estimates}
\label{subsec: error estimates}

This section estimates the error of discretizing Problem \eqref{eq: coupled state-adjoint} by the space-time interface-fitted method. To derive the desired estimates, we need several auxiliary results. 

First, consider a local high regularity function $w \in \Hs^s\left(Q_1 \cup Q_2\right)$ and denote by $w_i \in \Hs^s\left(Q_i\right)$ its restriction to $Q_i\ \left(i=1,2\right)$, where $s > 0$ is given. Recall from \cite[Theorem~2.30]{Ern2021} that if $\Gamma^\ast$ is Lipschitz continuous in $\mathbb{R}^{m+1}$ $\left(m=1 \text{ or } 2\right)$, then there exist smooth extensions $\Es_i: \Hs^s\left(Q_i\right)\rightarrow \Hs^s\left(Q_T\right)$ such that
\begin{equation}
    \label{eq: extension operator}
    \Es_i w=w_i \qq \text{in } Q_i, \qqqq \norm{\Es_i w}_{\Hs^s\left(Q_T\right)} \leq C\norm{w_i}_{\Hs^s\left(Q_i\right)}\qq (i=1,2).
\end{equation}

On the other hand, while the gradient of the solution to Problem \eqref{eq: coupled state-adjoint} is discontinuous across $\Gamma^\ast$, that of Problem \eqref{eq: discrete coupled state-adjoint} can only pose discontinuity across $\Gamma^\ast_h$. As a result, the space $\Hs^s\left(Q_{1,h} \cup Q_{2,h}\right)$ is more favorable than $\Hs^s\left(Q_1 \cup Q_2\right)$ when dealing with Problem \eqref{eq: discrete coupled state-adjoint}. To bridge these two spaces, let us introduce the following operator 
$$
\pi_s: \Hs^s\left(Q_1 \cup Q_2\right) \rightarrow \Hs^s\left(Q_{1,h} \cup Q_{2,h}\right),\qq w \mapsto \pi_s w := 
\begin{cases}
    \Es_1 w & \text{in}\q Q_{1,h}, \\
    \Es_2 w & \text{in}\q Q_{2,h}.
\end{cases}
$$
Clearly, $\pi_s w$ and $w$ coincide everywhere in $Q_T$ except on the mismatch region $S_h$. Regarding their difference within this region, we combine \cite[Theorems~3.3 and 3.2]{NLPT2024} and \cite[Lemma~3.3]{NLPT2024} to get the following lemma:

\begin{lemma}
    \label{lem: properties of pi}
    If $w\in \Hs^s\left(Q_1\cup Q_2\right)$ with $s\in \left[2, \frac{m+3}{2}\right]$, then we can find $h^\dagger >0$ such that for all $h\in \left(0, h^\dagger\right)$, the following estimate holds
    $$
    \norm{\kappa_h \nabla\left(\pi_s w\right) - \kappa \nabla w} _{\LLs^2\left(S_h\right)}\le C 
    \begin{Bmatrix}
        h\sqrt{\abs{\log h}}\\
        h^{\frac{2}{3}}
    \end{Bmatrix}\norm{w}_{\Hs^s\left(Q_1\cup Q_2\right)}\qqq \text{for }
    \begin{Bmatrix}
        m=1\\
        m=2
    \end{Bmatrix}.
    $$
    If $s\in \left(\frac{m+3}{2}, \infty\right)$, then we have
    $$
    \norm{\kappa_h \nabla\left(\pi_s w\right) - \kappa \nabla w} _{\LLs^2\left(S_h\right)}\le Ch \norm{w}_{\Hs^s\left(Q_1\cup Q_2\right)}.
    $$
    Assume that $\Gamma^\ast$ is a $\Cs^2$-continuous hypersurface in $\mathbb{R}^{m+1}$. For all $w\in \Hs^s\left(Q_1\cup Q_2\right)$ with $s\in \left[2,\infty\right)$, the function $\pi_s$ satisfies the estimate
    $$
    \norm{\Ds\left(\pi_s w- w\right)}_{\LLs^2\left(S_h\right)}\le C h \norm{w}_{\Hs^s\left(Q_1\cup Q_2\right)},
    $$
    where $\Ds:=\left(\nabla, \partial_t\right)^\top$ denotes the space-time gradient operator.
\end{lemma}

Next, we study the interpolation approximability. Consider $w\in \Hs^1\left(Q_T\right)\cap \Hs^s\left(Q_1\cup Q_2\right)$ with $s\in \left[2,\infty\right)$. It follows from \cite[Theorems~2.35 and 18.8]{Ern2021} that $w\in \Cs\left(\overline{Q_T}\right)$. Let $\mathcal{L}_h: \Cs\left(\overline{Q_T}\right)\rightarrow \Us_h$ be the Lagrange interpolation operator associated with $\Us_h$. For convenience, we define $e\left(w\right):= w - \mathcal{L}_h w$ and $e_{\pi}\left(w\right):= \pi_s w - \mathcal{L}_h w$. We have the following lemma:

\begin{lemma}
    \label{lem: interpolation error}
    For all $w\in \Hs^1\left(Q_T\right)\cap \Hs^s\left(Q_1\cup Q_2\right)$ with $s\in \left[2,\frac{m+3}{2}\right]$, the following estimate holds
    $$
    \norm{e\left(w\right)}_{\Ls^2\left(Q_T\right)} + h \norm{\Ds e\left(w\right)}_{\LLs^2\left(Q_T\right)} \le C 
    \begin{Bmatrix}
        h^2\sqrt{\abs{\log h}}\\
        h^{\frac{5}{3}}
    \end{Bmatrix}\norm{w}_{\Hs^s\left(Q_1\cup Q_2\right)}\ \text{for }
    \begin{Bmatrix}
        m=1\\
        m=2
    \end{Bmatrix}.
    $$
    If $s\in \left(\frac{m+3}{2}, \infty\right)$, then the estimate is optimal, i.e.,
    $$
    \norm{e\left(w\right)}_{\Ls^2\left(Q_T\right)} + h \norm{\Ds e\left(w\right)}_{\LLs^2\left(Q_T\right)} \le Ch^2\norm{w}_{\Hs^s\left(Q_1\cup Q_2\right)}.
    $$
    On the other hand, for all $w\in \Hs^1\left(Q_T\right)\cap \Hs^s\left(Q_1\cup Q_2\right)$ with $s\in \left[2,\infty\right)$, the operator $e_{\pi}$ satisfies the inequalities
    $$
    \norm{\Ds e_{\pi}\left(w\right)}_{\LLs^2\left(Q_{1,h}\cup Q_{2,h}\right)} \le C h \norm{w}_{\Hs^s\left(Q_1\cup Q_2\right)},
    $$
    and
    $$
    \vertiii{e_{\pi}\left(w\right)}_\ast \le C h \norm{w}_{\Hs^s\left(Q_1\cup Q_2\right)}.
    $$
\end{lemma}

\begin{proof}
    For the proof of the first inequality when $m=1$, please refer to \cite[Lemma~2.1]{CZ1998}. This result fundamentally relies on the Sobolev inequality
    $$
    \norm{w}_{\Ls^q\left(Q_T\right)}\le Cq^{\frac{1}{2}}\norm{w}_{\Hs^1\left(Q_T\right)}\qqqq \forall w\in \Hs^1\left(Q_T\right), \forall q\in \left[2,\infty\right),
    $$
    where the constant $C>0$ is independent of $q$. For the case of two spatial dimensions $\left(m=2\right)$, we employ analogous techniques to establish the corresponding interpolation error bound, except replacing the above inequality with the following one
    $$
    \norm{w}_{\Ls^q\left(Q_T\right)}\le C\norm{w}_{\Hs^1\left(Q_T\right)}\qqqq \forall w\in \Hs^1\left(Q_T\right), \forall q\in \left[1,6\right].
    $$
    
    The second estimate is proved in \cite[Lemma~5.2]{MHH2025}. The third inequality is a direct extension of \cite[Lemma~3.4]{NLPT2024}. 
    
    Now, we prove the last one. By using \eqref{eq: discrete riesz representation} and the third inequality, we have
    \begin{align*}
        \vertiii{e_{\pi}\left(w\right)}^2_\ast &= \vertiii{e_{\pi}\left(w\right)}^2 + \vertiii{z_h\left(\partial_t\left( e_{\pi}\left(w\right)\right)\right)}^2 \\
        & \le C\left(\norm{\nabla e_{\pi}\left(w\right)}^2_{\LLs^2\left(Q_{1,h}\cup Q_{2,h}\right)} + \norm{\partial_t \left(e_{\pi}\left(w\right)\right)}^2_{\Ls^2\left(Q_{1,h}\cup Q_{2,h}\right)}\right)\\
        & \le C\norm{\Ds e_{\pi}\left(w\right)}^2_{\LLs^2\left(Q_{1,h}\cup Q_{2,h}\right)} \le C h^2\norm{w}^2_{\Hs^s\left(Q_1\cup Q_2\right)},
    \end{align*}
    which gives us the desired result.
\end{proof}

We are now in a position to establish the main result of this section. Due to the favorable properties of $\vertiii{\cdot}$, $\vertiii{\cdot}_{\ast}$, and $\pi_s$ for handling discrete problems, it is reasonable to choose estimating the error $\vertiii{\pi_s \overline{u} - \overline{u}_h}^2_\ast + \vertiii{\pi_s \overline{p} - \overline{p}_h}^2$. We have the following theorem: 

\begin{theorem}
    \label{theo: error estimate}
    Let $\left(\overline{u},\overline{p}\right)\in \Us_0\times \Us_T$ and $\left(\overline{u}_h,\overline{p}_h\right)\in \Us_h\times \Ws_h$ be the solutions to Problems \eqref{eq: coupled state-adjoint} and \eqref{eq: discrete coupled state-adjoint}, respectively. If $\overline{u}, \overline{p}\in \Hs^1\left(Q_T\right)\cap \Hs^s\left(Q_1\cup Q_2\right)$ with $s\in \left[2, \frac{m+3}{2}\right]$, then there exist the constants $C>0$ and $h^\dagger >0$ such that for all $h\in \left(0, h^\dagger\right)$, we have the following estimate
    \begin{align*}
        & \sqrt{\vertiii{\pi_s \overline{u} - \overline{u}_h}^2_\ast + \vertiii{\pi_s \overline{p} - \overline{p}_h}^2} \le \\
        & \q \le C 
    \begin{Bmatrix}
        h\sqrt{\abs{\log h}}\\
        h^{\frac{2}{3}}
    \end{Bmatrix}\left(\norm{\overline{u}}_{\Hs^s\left(Q_1\cup Q_2\right)} + \norm{\overline{p}}_{\Hs^s\left(Q_1\cup Q_2\right)}\right)\qqq \text{for }
    \begin{Bmatrix}
        m=1\\
        m=2
    \end{Bmatrix}.
    \end{align*}
    If $s\in \left(\frac{m+3}{2}, \infty\right)$, then the following inequality holds
    $$
    \sqrt{\vertiii{\pi_s \overline{u} - \overline{u}_h}^2_\ast + \vertiii{\pi_s \overline{p} - \overline{p}_h}^2} \le Ch \left(\norm{\overline{u}}_{\Hs^s\left(Q_1\cup Q_2\right)} + \norm{\overline{p}}_{\Hs^s\left(Q_1\cup Q_2\right)}\right).
    $$
\end{theorem}

\begin{proof}
    Using the triangle inequality and \Cref{lem: interpolation error}, we have
    \begin{align*}
        & \vertiii{\pi_s \overline{u} - \overline{u}_h}^2_\ast + \vertiii{\pi_s \overline{p} - \overline{p}_h}^2 \le\\
        & \q \le 2\left(\vertiii{e_{\pi}\left(\overline{u}\right)}^2_\ast + \vertiii{\mathcal{L}_h\overline{u} - \overline{u}_h}^2_\ast\right) + 2\left(\vertiii{e_{\pi}\left(\overline{p}\right)}^2 + \vertiii{\mathcal{L}_h\overline{p} - \overline{p}_h}^2\right)\\
        & \q \le Ch^2\norm{\overline{u}}^2_{\Hs^s\left(Q_1\cup Q_2\right)} + 2\vertiii{\mathcal{L}_h\overline{u} - \overline{u}_h}^2_\ast +  Ch^2\norm{\overline{p}}^2_{\Hs^s\left(Q_1\cup Q_2\right)} + 2\vertiii{\mathcal{L}_h\overline{p} - \overline{p}_h}^2.
    \end{align*}
    Due to \Cref{lem: discrete bnb1}, the following estimate holds
    \begin{align*}
        & C\sqrt{\vertiii{\mathcal{L}_h\overline{u} - \overline{u}_h}^2_\ast + \vertiii{\mathcal{L}_h\overline{p} - \overline{p}_h}^2} \le \\
        & \q \le \sup_{\left(\vphi_h, \phi_h\right) \in \Ws_h\times \Us_h\setminus \left\{\left(0_{\Ws_h},0_{\Us_h}\right)\right\}} \dfrac{\mathcal{A}_h\left(\left(\mathcal{L}_h\overline{u} - \overline{u}_h, \mathcal{L}_h\overline{p} - \overline{p}_h\right), \left(\vphi_h, \phi_h\right)\right)}{\sqrt{\vertiii{\vphi_h}^2 + \vertiii{\phi_h}^2_\ast}}\\
        & \q = \sup_{\left(\vphi_h, \phi_h\right) \in \Ws_h\times \Us_h\setminus \left\{\left(0_{\Ws_h},0_{\Us_h}\right)\right\}} \dfrac{\mathcal{A}_h\left(\left(\mathcal{L}_h\overline{u}, \mathcal{L}_h\overline{p}\right), \left(\vphi_h, \phi_h\right)\right) - \mathcal{A}\left(\left(\overline{u}, \overline{p}\right), \left(\vphi_h, \phi_h\right)\right)}{\sqrt{\vertiii{\vphi_h}^2 + \vertiii{\phi_h}^2_\ast}}\\
        & \q = \sup_{\left(\vphi_h, \phi_h\right) \in \Ws_h\times \Us_h\setminus \left\{\left(0_{\Ws_h},0_{\Us_h}\right)\right\}} \dfrac{I_1 - I_2}{\sqrt{\vertiii{\vphi_h}^2 + \vertiii{\phi_h}^2_\ast}},
    \end{align*}
    where $I_1$ and $I_2$ are defined as 
    \begin{align*}
        I_1 &:= \sum_{i=1}^2\int\limits_{Q_{i,h}} \partial_t \left(\Es_i \overline{u}\right) \vphi_h + \left[\vb\cdot \nabla \left(\Es_i \overline{u}\right)\right]\vphi_h +  \kappa_i \nabla \left(\Es_i \overline{u}\right) \cdot \nabla \vphi_h \dx \dt\\
        &\q +\dfrac{1}{\eta} \sum_{i=1}^2\int\limits_{Q_{i,h}}\kappa_i \nabla \left(\Es_i \overline{p}\right) \cdot \nabla \vphi_h \dx \dt + \int\limits_{Q_T}\overline{u}\phi_h\dx\dt - \mathcal{A}\left(\left(\overline{u}, \overline{p}\right), \left(\vphi_h, \phi_h\right)\right)\\
        &\q - \sum_{i=1}^2\int\limits_{Q_{i,h}} -\partial_t \left(\Es_i \overline{p}\right) \phi_h - \left[\vb\cdot \nabla \left(\Es_i \overline{p}\right)\right]\phi_h +  \kappa_i \nabla \left(\Es_i \overline{p}\right) \cdot \nabla \phi_h \dx \dt,
    \end{align*}
    and
    \begin{align*}
        I_2 &:= \sum_{i=1}^2\int\limits_{Q_{i,h}} \partial_t \left(\Es_i \overline{u}\right) \vphi_h + \left[\vb\cdot \nabla \left(\Es_i \overline{u}\right)\right]\vphi_h +  \kappa_i \nabla \left(\Es_i \overline{u}\right) \cdot \nabla \vphi_h \dx \dt\\
        &\q +\dfrac{1}{\eta} \sum_{i=1}^2\int\limits_{Q_{i,h}}\kappa_i \nabla \left(\Es_i \overline{p}\right) \cdot \nabla \vphi_h \dx \dt + \int\limits_{Q_T}\overline{u}\phi_h\dx\dt - \mathcal{A}_h\left(\left(\mathcal{L}_h\overline{u}, \mathcal{L}_h\overline{p}\right), \left(\vphi_h, \phi_h\right)\right)\\
        &\q - \sum_{i=1}^2\int\limits_{Q_{i,h}} -\partial_t \left(\Es_i \overline{p}\right) \phi_h - \left[\vb\cdot \nabla \left(\Es_i \overline{p}\right)\right]\phi_h +  \kappa_i \nabla \left(\Es_i \overline{p}\right) \cdot \nabla \phi_h \dx \dt.
    \end{align*}
    
    We first consider the case where $\overline{u}, \overline{p}\in \Hs^1\left(Q_T\right)\cap \Hs^s\left(Q_1\cup Q_2\right)$ for some $s\in \left[2,\frac{m+3}{2}\right]$. By applying the arguments for \eqref{eq: relation between adjoint and state}, we have $a\left(\phi_h, \overline{p}\right) = a^\prime\left(\overline{p},\phi_h\right)$ for all $\phi_h \in \Us_h$, which allows us to rewrite $I_1$ as follows
    \begin{align*}
        I_1 &= \sum_{i=1}^2\int\limits_{Q_{i,h}} \partial_t \left(\Es_i \overline{u}\right) \vphi_h + \left[\vb\cdot \nabla \left(\Es_i \overline{u}\right)\right]\vphi_h +  \kappa_i \nabla \left(\Es_i \overline{u}\right) \cdot \nabla \vphi_h \dx \dt - a\left(\overline{u}, \vphi_h\right)\\
        & \q +\dfrac{1}{\eta}\left(\sum_{i=1}^2\int\limits_{Q_{i,h}}\kappa_i \nabla \left(\Es_i \overline{p}\right) \cdot \nabla \vphi_h \dx \dt - \int\limits_{Q_T}\kappa \nabla \overline{p}\cdot \nabla \vphi_h \dx \dt\right)\\
        & \q - \sum_{i=1}^2\int\limits_{Q_{i,h}} -\partial_t \left(\Es_i \overline{p}\right) \phi_h - \left[\vb\cdot \nabla \left(\Es_i \overline{p}\right)\right]\phi_h +  \kappa_i \nabla \left(\Es_i \overline{p}\right) \cdot \nabla \phi_h \dx \dt + a^\prime\left(\overline{p},\phi_h\right)\\
        & = \int\limits_{S_h} \partial_t\left(\pi_s \overline{u} - \overline{u}\right)\vphi_h + \left[\vb\cdot \nabla \left(\pi_s \overline{u} - \overline{u}\right)\right]\vphi_h +  \left[\kappa_h \nabla \left(\pi_s \overline{u}\right) - \kappa\nabla \overline{u}\right] \cdot \nabla \vphi_h \dx \dt\\
        & \q +\dfrac{1}{\eta} \int\limits_{S_h} \left[\kappa_h \nabla \left(\pi_s \overline{p}\right) - \kappa\nabla \overline{p}\right] \cdot \nabla \vphi_h \dx \dt\\
        & \q - \int\limits_{S_h} -\partial_t\left(\pi_s \overline{p} - \overline{p}\right)\phi_h - \left[\vb\cdot \nabla \left(\pi_s \overline{p} - \overline{p}\right)\right]\phi_h +  \left[\kappa_h \nabla \left(\pi_s \overline{p}\right) - \kappa\nabla \overline{p}\right] \cdot \nabla \phi_h \dx \dt.
    \end{align*}
    Similarly, we integrate by parts with $\phi_h\in \Us_h$ and $\mathcal{L}_h\overline{p}\in \Us_h$, then utilize the technique in \eqref{eq: optimality conditions 5} to obtain
    \begin{align*}
        a^\prime_h\left(\phi_h, \mathcal{L}_h  \overline{p}\right) &=\int\limits_{\Omega}\left(\mathcal{L}_h\overline{p}\right)\left(\xb, T\right)\phi_h\left(\xb, T\right)\dx\\
        & \q + \int\limits_{Q_T} -\partial_t \left(\mathcal{L}_h\overline{p}\right) \phi_h - \left[\vb\cdot \nabla \left(\mathcal{L}_h\overline{p}\right)\right]\phi_h +  \kappa_i \nabla \left(\mathcal{L}_h\overline{p}\right) \cdot \nabla \phi_h \dx \dt.
    \end{align*}
    Therefore, we rewrite $I_2$ into the following form
    \begin{align*}
        I_2 &= \sum_{i=1}^2\int\limits_{Q_{i,h}} \partial_t \left(e_{\pi}\left(\overline{u}\right)\right) \vphi_h + \left[\vb\cdot \nabla e_{\pi}\left(\overline{u}\right)\right]\vphi_h +  \kappa_i \nabla e_{\pi}\left(\overline{u}\right) \cdot \nabla \vphi_h \dx \dt\\
        & \q +\dfrac{1}{\eta} \sum_{i=1}^2\int\limits_{Q_{i,h}}\kappa_i \nabla e_{\pi}\left(\overline{p}\right) \cdot \nabla \vphi_h \dx \dt + \int\limits_{Q_T}e \left(\overline{u}\right) \phi_h\dx\dt\\
        & \q - \sum_{i=1}^2\int\limits_{Q_{i,h}} -\partial_t \left(\Es_i \overline{p}\right) \phi_h - \left[\vb\cdot \nabla \left(\Es_i \overline{p}\right)\right]\phi_h +  \kappa_i \nabla \left(\Es_i \overline{p}\right) \cdot \nabla \phi_h \dx \dt \\
        & \q + \int\limits_{\Omega}\left(\mathcal{L}_h\overline{p}\right)\left(\xb, T\right)\phi_h\left(\xb, T\right)\dx\\
        & \q + \int\limits_{Q_T} -\partial_t \left(\mathcal{L}_h\overline{p}\right) \phi_h - \left[\vb\cdot \nabla \left(\mathcal{L}_h\overline{p}\right)\right]\phi_h +  \kappa_i \nabla \left(\mathcal{L}_h\overline{p}\right) \cdot \nabla \phi_h \dx \dt\\
        & = \sum_{i=1}^2\int\limits_{Q_{i,h}} \partial_t \left(e_{\pi}\left(\overline{u}\right)\right)\vphi_h + \left[\vb\cdot \nabla e_{\pi}\left(\overline{u}\right)\right]\vphi_h +  \kappa_i \nabla e_{\pi}\left(\overline{u}\right) \cdot \nabla \vphi_h \dx \dt\\
        & \q +\dfrac{1}{\eta} \sum_{i=1}^2\int\limits_{Q_{i,h}}\kappa_i \nabla e_{\pi}\left(\overline{p}\right) \cdot \nabla \vphi_h \dx \dt + \int\limits_{Q_T}e\left(\overline{u}\right)\phi_h\dx\dt\\
        & \q - \sum_{i=1}^2\int\limits_{Q_{i,h}} -\partial_t \left(e_{\pi}\left(\overline{p}\right)\right) \phi_h - \left[\vb\cdot \nabla e_{\pi}\left(\overline{p}\right)\right]\phi_h +  \kappa_i \nabla e_{\pi}\left(\overline{p}\right) \cdot \nabla \phi_h \dx \dt \\
        & \q + \int\limits_{\Omega}\left(\mathcal{L}_h\overline{p}\right)\left(\xb, T\right)\phi_h\left(\xb, T\right)\dx.
    \end{align*}
    By employing \eqref{eq: compare L2 and the W norm}, together with \Cref{lem: properties of pi,lem: interpolation error}, we hence imply
    \begin{align*}
        I_1 - I_2&\le C\left\{\norm{\Ds\left(\pi_s\overline{u}-\overline{u}\right)}_{\LLs^2\left(S_h\right)} + \norm{\Ds\left(\pi_s\overline{p}-\overline{p}\right)}_{\LLs^2\left(S_h\right)} + \norm{\kappa_h \nabla\left(\pi_s \overline{u}\right) - \kappa \nabla \overline{u}}_{\LLs^2\left(S_h\right)}\right.\\
        &\q +\norm{\kappa_h \nabla\left(\pi_s \overline{p}\right) - \kappa \nabla \overline{p}}_{\LLs^2\left(S_h\right)} + \norm{e\left(\overline{u}\right)}_{\Ls^2\left(Q_T\right)} + \norm{\Ds e_{\pi}\left(\overline{u}\right)}_{\LLs^2\left(Q_{1,h}\cup Q_{2,h}\right)}\\
        & \q + \left. \norm{\Ds e_{\pi}\left(\overline{p}\right)}_{\LLs^2\left(Q_{1,h}\cup Q_{2,h}\right)}\right\}\left(\norm{\nabla\vphi_h}_{\LLs^2\left(Q_T\right)} + \norm{\nabla\phi_h}_{\LLs^2\left(Q_T\right)}\right) \\
        & \q - \int\limits_{\Omega} \left(\mathcal{L}_h\overline{p}\right)\left(\xb, T\right)\phi_h\left(\xb, T\right)\dx\\
        & \le C\begin{Bmatrix}
            h\sqrt{\abs{\log h}}\\
            h^{\frac{2}{3}}
        \end{Bmatrix}\left(\norm{\overline{u}}_{\Hs^s\left(Q_1\cup Q_2\right)} + \norm{\overline{p}}_{\Hs^s\left(Q_1\cup Q_2\right)}\right)\left(\vertiii{\vphi_h} + \vertiii{\phi_h}_\ast\right)\\
        & \q - \int\limits_{\Omega} \left(\mathcal{L}_h\overline{p}\right)\left(\xb, T\right)\phi_h\left(\xb, T\right)\dx \qqqq \qqqq \q \text{for }
        \begin{Bmatrix}
            m=1\\
            m=2
        \end{Bmatrix}.
    \end{align*}
    Now, we proceed with the last term on the right-hand side of this inequality. Let $\mathcal{T}_{h, T}$ denote the set of all elements intersecting the upper boundary $\Omega\times \left\{T\right\}$ of $Q_T$. On the one hand, by the fact that $\overline{p}\in \Us_T$, the multiplicative trace inequality \cite[Lemma~12.15]{Ern2021}, and \Cref{lem: interpolation error}, we have
    \begin{align*}
        \norm{\mathcal{L}_h\left(\overline{p}\right)\left(\cdot, T\right)}^2_{\Ls^2\left(\Omega\right)} & = \norm{e\left(\overline{p}\right)\left(\cdot, T\right)}^2_{\Ls^2\left(\Omega\right)} \\
        & \le \sum_{K\in \mathcal{T}_{h, T}}\norm{\gamma_K \left(e\left(\overline{p}\right)\right)}^2_{\Ls^2\left(\partial K\right)}\\
        & \le C\sum_{K\in \mathcal{T}_{h, T}} \norm{e\left( \overline{p}\right)}_{\Ls^2\left(K\right)}\left(h^{-1} \norm{e\left(\overline{p}\right)}_{\Ls^2\left(K\right)} + \norm{\Ds e\left(\overline{p}\right)}_{\LLs^2\left(K\right)}\right)\\
        & \le C 
        \begin{Bmatrix}
            h^3\abs{\log h}\\
            h^{\frac{7}{3}}
        \end{Bmatrix}\norm{\overline{p}}^2_{\Hs^s\left(Q_1\cup Q_2\right)}\qqqq \qq \text{for }
        \begin{Bmatrix}
            m=1\\
            m=2
        \end{Bmatrix},
    \end{align*}
    where $\gamma_K: \Hs^1\left(K\right)\rightarrow \Ls^2\left(\partial K\right)$ denotes the trace operator defined on $K$. On the other hand, it follows from the discrete trace inequality \cite[Lemma~12.8]{Ern2021} that
    \begin{align*}
        \norm{\phi_h\left(\cdot, T\right)}^2_{\Ls^2\left(\Omega\right)}&\le \sum_{K\in \mathcal{T}_{h, T}}\norm{\gamma_K \left(\phi_h\right)}^2_{\Ls^2\left(\partial K\right)}\\
        & \le C \sum_{K\in \mathcal{T}_{h, T}} h^{-1} \norm{\phi_h}^2_{\Ls^2\left(K\right)} \le C h^{-1}\norm{\phi_h}^2_{\Ls^2\left(Q_T\right)}.
    \end{align*}
    By combining the last two inequalities with \eqref{eq: compare L2 and the W norm}, we thus obtain 
    $$
    \abs{\int\limits_{\Omega}\mathcal{L}_h\left(\overline{p}\right)\left(\xb, T\right)\phi_h\left(\xb, T\right)\dx} \le C 
    \begin{Bmatrix}
        h\sqrt{\abs{\log h}}\\
        h^{\frac{2}{3}}
    \end{Bmatrix}\norm{\overline{p}}_{\Hs^s\left(Q_1\cup Q_2\right)}\vertiii{\phi_h}_\ast\q\text{for }
    \begin{Bmatrix}
        m=1\\
        m=2
    \end{Bmatrix}.
    $$
    Therefore, we imply that 
    \begin{align*}
        I_1 -I_2 &\le C 
        \begin{Bmatrix}
            h\sqrt{\abs{\log h}}\\
            h^{\frac{2}{3}}
        \end{Bmatrix}\left(\norm{\overline{u}}_{\Hs^s\left(Q_1\cup Q_2\right)} + \norm{\overline{p}}_{\Hs^s\left(Q_1\cup Q_2\right)}\right)\left(\vertiii{\vphi_h} + \vertiii{\phi_h}_\ast\right)\\
    & \le C 
        \begin{Bmatrix}
            h\sqrt{\abs{\log h}}\\
            h^{\frac{2}{3}}
        \end{Bmatrix}\left(\norm{\overline{u}}_{\Hs^s\left(Q_1\cup Q_2\right)} + \norm{\overline{p}}_{\Hs^s\left(Q_1\cup Q_2\right)}\right)\sqrt{\vertiii{\vphi_h}^2 + \vertiii{\phi_h}^2_\ast},
    \end{align*}
    which leads to
    $$
    \sqrt{\vertiii{\mathcal{L}_h\overline{u} - \overline{u}_h}^2_\ast + \vertiii{\mathcal{L}_h\overline{p} - \overline{p}_h}^2}\le C 
        \begin{Bmatrix}
            h\sqrt{\abs{\log h}}\\
            h^{\frac{2}{3}}
        \end{Bmatrix}\left(\norm{\overline{u}}_{\Hs^s\left(Q_1\cup Q_2\right)} + \norm{\overline{p}}_{\Hs^s\left(Q_1\cup Q_2\right)}\right).
    $$
    This inequality, along with all previous estimates which derived from the first estimate in \Cref{lem: properties of pi} and the first estimate in \Cref{lem: interpolation error}, holds for sufficiently small $h>0$. The conclusion follows. 

    In the case where $\overline{u}, \overline{p}\in \Hs^1\left(Q_T\right)\cap \Hs^s\left(Q_1\cup Q_2\right)$ with $s\in \left(\frac{m+3}{2}, \infty\right)$, the argument proceeds similarly to yield the desired result. However, in this setting, we invoke the corresponding estimates from \Cref{lem: properties of pi,lem: interpolation error} that apply to functions with higher local regularity.
\end{proof}

Finally, we note that the error estimates in \Cref{theo: error estimate} can also be derived in alternative forms. For instance, by combining \Cref{lem: properties of pi,theo: error estimate}, we obtain the following result:

\begin{corollary}
    \label{cor: usual norm error estimate}
    If $\overline{u}, \overline{p}\in \Hs^1\left(Q_T\right)\cap \Hs^s\left(Q_1\cup Q_2\right)$ for some $s\in \left[2, \frac{m+3}{2}\right]$, then we can find $h^\dagger >0$ such that for all $h\in \left(0, h^\dagger\right)$, the following estimate holds
    \begin{align*}
        & \sqrt{\vertiii{\overline{u} - \overline{u}_h}^2_\ast + \vertiii{\overline{p} - \overline{p}_h}^2} \le \\
        & \q \le C 
    \begin{Bmatrix}
        h\sqrt{\abs{\log h}}\\
        h^{\frac{2}{3}}
    \end{Bmatrix}\left(\norm{\overline{u}}_{\Hs^s\left(Q_1\cup Q_2\right)} + \norm{\overline{p}}_{\Hs^s\left(Q_1\cup Q_2\right)}\right)\qqq \text{for }
    \begin{Bmatrix}
        m=1\\
        m=2
    \end{Bmatrix}.
    \end{align*}
    On the other hand, if $s\in \left(\frac{m+3}{2}, \infty\right)$, then we have the inequality
    $$
    \sqrt{\vertiii{\overline{u} - \overline{u}_h}^2_\ast + \vertiii{\overline{p} - \overline{p}_h}^2} \le Ch \left(\norm{\overline{u}}_{\Hs^s\left(Q_1\cup Q_2\right)} + \norm{\overline{p}}_{\Hs^s\left(Q_1\cup Q_2\right)}\right).
    $$
\end{corollary}

\section{Numerical results} 
\label{sec: numerical results} 

In this section, we present several numerical examples to validate our theoretical results. All computations are performed using FreeFEM++ \cite{Hecht2012}. For the implementation, the trial and test spaces in Problem \eqref{eq: discrete coupled state-adjoint} are both chosen as $\Us_h\times \Us_h$. Consequently, both the state and adjoint are approximated by continuous, space-time piecewise linear finite elements with zero initial values. 

\subsection{Example 1} 

We first start with an example in one spatial dimension. Consider the space-time cylinder $Q_T = \left(0,1\right)^2$ comprising two subdomains
$$
Q_1 = \left\{\left(x,t\right)\in Q_T\mid 0.4 + \mathrm{s}\left(t\right) < x < 0.6 + \mathrm{s}\left(t\right), \, 0<t<1\right\}, \qq Q_2 = Q_T\setminus \overline{Q_1},
$$
where $\mathrm{s}\left(t\right)$ is defined as
$$
\mathrm{s}\left(t\right) := \int\limits_0^t \mathrm{v}\left(\tau\right) \di \tau,
$$
and $\mathrm{v}$ is the velocity of the interface $\Gamma\left(t\right)$. The diffusion coefficient $\kappa$ is taken as $\left(\kappa_1, \kappa_2\right) =  \left(0.5,1\right)$. We set $\eta = 10^{-6}$ and define
$$
w_1\left(x,t\right) := \sin\paren{20\pi \paren{x - \mathrm{s}\left(t\right)} - \dfrac{47\pi}{6}} + \sin \paren{10\pi \mathrm{s}\left(t\right) + \dfrac{23\pi}{6}},
$$
and
$$
w_2\left(x,t\right) := \sin\paren{10\pi \paren{x- \mathrm{s}\left(t\right)} - \dfrac{23\pi}{6}} + \sin \paren{10\pi \mathrm{s}\left(t\right) + \dfrac{23\pi}{6}}.
$$
The optimal state $u^\ast\in \Us_0$ and adjoint $p^\ast\in \Us_T$ are then prescribed as
$$
u^\ast\left(x,t\right) = \begin{cases}
    w_1\left(x,t\right) \sin \paren{\dfrac{\pi t}{2}} & \text{for }\left(x, t\right) \in Q_1, \\[10pt]
    w_2\left(x,t\right) \sin \paren{\dfrac{\pi t}{2}} & \text{for } \left(x, t\right) \in Q_2,
\end{cases}
$$
and
$$
p^\ast\left(x, t\right) = \begin{cases}
    -\eta w_1\left(x,t\right) \sin \paren{\dfrac{\pi -\pi t}{2}} & \text{for } \left(x, t\right) \in Q_1, \\[10pt]
    -\eta w_2\left(x,t\right) \sin \paren{\dfrac{\pi -\pi t}{2}} & \text{for } \left(x, t\right) \in Q_2,
\end{cases}
$$
which exhibit the regularity $u^\ast, p^\ast\in \Hs^1\left(Q_T\right)\cap \Hs^3\left(Q_1\cup Q_2\right)$. The desired state $u_d$ is determined accordingly. Finally, we use the following error metric
$$
\mathcal{E}:= \left(\,\,\int\limits_{Q_T}\abs{\nabla\left(u^\ast - \overline{u}_h\right)}^2 + \abs{\nabla\left(p^\ast - \overline{p}_h\right)}^2\dx\dt\right)^{1/2},
$$
where $\left(\overline{u}_h, \overline{p}_h\right)\in \Us_h\times \Us_h$ is the solution to Problem \eqref{eq: discrete coupled state-adjoint}. It is evident that $\mathcal{E}$ is simpler to implement while maintaining the same convergence order as the estimates in \Cref{theo: error estimate,cor: usual norm error estimate}.

\begin{figure}
\centering
    \subfigure{\begin{tikzpicture}
        \draw[-latex,line width=0.2mm] (0,0)--(5,0) node[below=1mm]{$x$}; 
        \draw[-latex,line width=0.2mm] (0,0)--(0,4.7) node[left]{$t$};
        \node at (0, 0) [below=0.8mm]{$0$};
        \node at (4, 0) [below=0.6mm]{$1$};
        \node at (0,4) [left]{$1$};
        \draw[line width = 0.4mm,black] (0, 0) -- (4, 0) -- (4, 4) -- (0, 4) -- (0, 0);
        \draw[domain=0:4,smooth,variable=\y,line width = 0.3mm,black] plot ({1.6 + 0*\y}, {\y});
        \draw[domain=0:4,smooth,variable=\y,line width = 0.3mm,black] plot ({2.4 + 0*\y}, {\y});
        \node at (2, 2) {$Q_1$};
        \node at (1, 1) {$Q_2$};
        \node at (4.5, 2) {$Q_T$};
        \node at (2.8, 3) {$\Gm^{\ast}$};
        \node at (1.3, 3) {$\Gm^{\ast}$};
    \end{tikzpicture}}
    \subfigure{\begin{tikzpicture}
        \draw[-latex,line width=0.2mm] (0,0)--(5,0) node[below=1mm]{$x$}; 
        \draw[-latex,line width=0.2mm] (0,0)--(0,4.7) node[left]{$t$};
        \node at (0, 0) [below=0.8mm]{$0$};
        \node at (4, 0) [below=0.6mm]{$1$};
        \node at (0,4) [left]{$1$};
        \draw[line width = 0.4mm,black] (0, 0) -- (4, 0) -- (4, 4) -- (0, 4) -- (0, 0);
        \draw[domain=0:4,smooth,variable=\y,line width = 0.3mm,black] plot ({1.8 - 0.2*cos(deg(2*pi*\y/4))}, {\y});
        \draw[domain=0:4,smooth,variable=\y,line width = 0.3mm,black] plot ({2.6 - 0.2*cos(deg(2*pi*\y/4))}, {\y});
        \node at (2.4, 2) {$Q_1$};
        \node at (1, 1) {$Q_2$};
        \node at (4.5, 2) {$Q_T$};
        \node at (3, 3) {$\Gm^{\ast}$};
        \node at (1.45, 3) {$\Gm^{\ast}$};
    \end{tikzpicture}}
    \caption{The space-time cylinder $Q_T$ in Example 1, considering two different cases: $\mathrm{v} = 0$ (left) and $\mathrm{v} =0.1 \pi \sin\left(2 \pi t\right)$ (right).}
    \label{fig:domain_2D}
\end{figure}
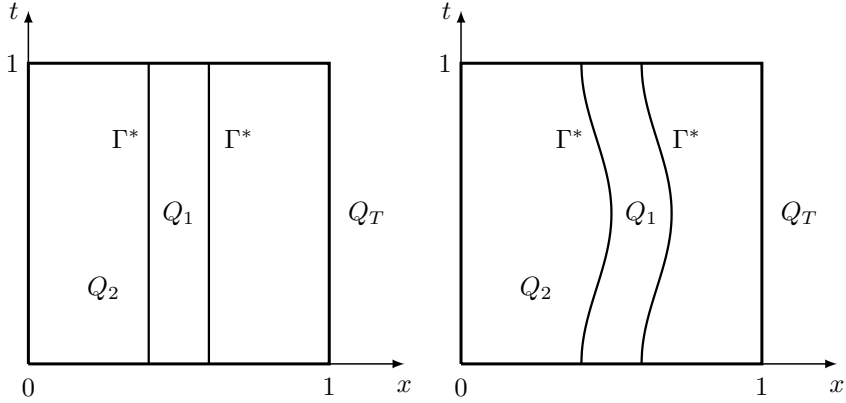

We begin by examining the case where $\mathrm{v} = 0$, corresponding to a non-moving interface problem. The space-time cylinder $Q_T$ is illustrated in \Cref{fig:domain_2D} (left). The error and experimental order of convergence are provided in \Cref{tab:example1a}, exhibiting a linear convergence of $\mathcal{E}$. This outcome is consistent with the results established in \Cref{theo: error estimate,cor: usual norm error estimate}, thereby validating the convergence behavior of our method for optimal control problems with a stationary interface.

Subsequently, we consider a time-dependent velocity $\mathrm{v} = 0.1 \pi \sin\left(2 \pi t\right)$, which results in a curved space-time interface, as depicted in \Cref{fig:domain_2D} (right). The error values and corresponding convergence rates for various mesh sizes are presented in \Cref{tab:example1b}. As the mesh is refined, we observe that $\mathcal{E}$ decreases at a rate of $\mathcal{O}\left(h\right)$. Clearly, the convergence order in both cases remains optimal regardless of the velocity selection, aligning with the established theoretical results.

\begin{table}[!htp]
    \centering
    \caption{Errors for various levels of mesh refinement in Example 1 when $\mathrm{v} = 0$.} 
    \label{tab:example1a}
    \begin{tabular}{ r c c c } 
        \toprule
        $\#$Dofs & Mesh size $h$  & Error $\mathcal{E}$& Order \\ \hline \hline
        3,394 &  $3.358 \times 10^{-2}$ &$4.732$ & --   \\ 
        13,963 & $1.684 \times 10^{-2}$ &  $2.278$ &1.054 \\ 
        53,114 &  $9.829 \times 10^{-3}$ & $1.152$& 0.968   \\
        120,956 & $6.112 \times 10^{-3}$ & $0.757$&1.068\\
        222,642 & $4.246 \times 10^{-3}$  & $0.557$& 1.194\\
        482,922 & $3.109 \times 10^{-3}$ & $0.373$ & 1.118\\
        \bottomrule
    \end{tabular}
\end{table}

\begin{table}[!htp]
    \centering
    \caption{Errors for various levels of mesh refinement in Example 1 when $\mathrm{v} = 0.1 \pi \sin\left(2 \pi t\right)$.} \label{tab:example1b}
    \begin{tabular}{ r c c c } 
        \toprule
        $\#$Dofs & Mesh size $h$  & Error $\mathcal{E}$& Order \\ \hline \hline
        3,255 &  $3.265 \times 10^{-2}$ &$4.947$ & -- \\ 
        13,257 & $1.633 \times 10^{-2}$ & $2.362$ & $1.065$\\ 
        52,011 &  $9.027 \times 10^{-3}$ & $1.188$  & $1.049$   \\
        81,095 & $7.123 \times 10^{-3}$  & $0.950$&$1.003$\\
        155,221 & $5.874 \times 10^{-3}$  & $0.685$ &$1.007$\\
        257,233 & $4.109 \times 10^{-3}$  & $0.530$& $1.021$\\
        \bottomrule
    \end{tabular}
\end{table}

\subsection{Example 2}

In this example, we consider Problem \eqref{eq: problem formulation} in the setting where $\Omega = \left\{\xb = \left(x, y\right)^\transpose\in \mathbb{R}^2\mid x^2 + y^2 \le \frac{1}{4}\right\}$ and $T=1$. The interface $\Gamma\left(t\right)$ is assumed to rotate about the $t$-axis throughout the time interval, with a velocity field given by $\vb = \left(-2 \pi y, 2 \pi x\right)^\transpose$. The discretized space-time cylinder $Q_T$ is depicted in \Cref{fig:model_3D}. 

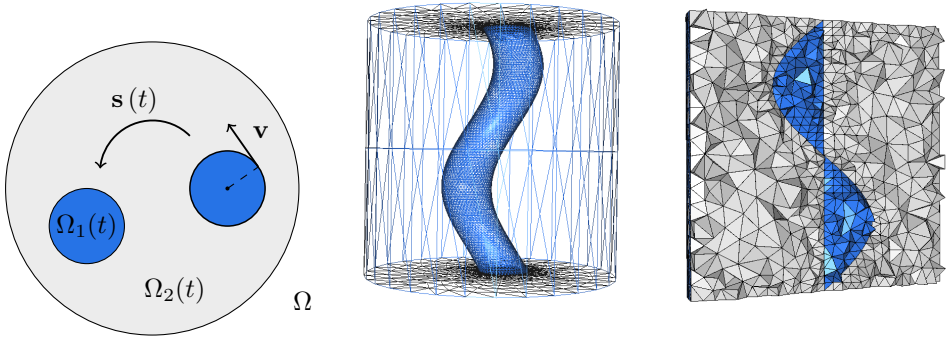
\begin{figure}[!htp]
    \subfigure{\begin{tikzpicture}
        \def\R{0.5}
        \def\Ri{1.95}
        \def\anga{60}
        \def\angaa{35}
        \coordinate (O) at (0,0);
        \coordinate (O1) at (1,0);
        \coordinate (R1) at (\anga:\Ri);
        \coordinate (X1) at (\angaa:{\Ri/cos(\anga-\angaa)});
        \coordinate (O2) at (-0.87,-0.5);
        \draw[fill = gray!15] (0,0) circle (1.95cm);
        \definecolor{myblue}{rgb}{0.15, 0.45, 0.9}
        \draw[fill = myblue, opacity=1.0, line width = 0.2mm] (O1) circle (\R);
        \draw[fill = black] (1,0) circle (0.02cm);
        \draw[dashed] (O1) -- +(\angaa:\R);
        \draw[->, line width = 0.25mm] (O1) ++(\angaa:\R) -- ++(\angaa + 90: 0.7) node[midway,right, yshift = 5pt] {$\vb$};
        \draw[fill = myblue, opacity=1.0] (O2) circle(\R);
        \draw[->, line width = 0.25mm] (0.5,0.7) arc(45:175:0.7) node[midway,above] {$\mathrm{\textbf{s}}\left(t\right)$};
        \node at (O2) {$\Om_1(t)$};
        \node at (0.3,-1.35) {$\Om_2(t)$};
        \node at (2,-1.5) {$\Om$};
    \end{tikzpicture}} 
    \subfigure{\includegraphics[scale = .2]{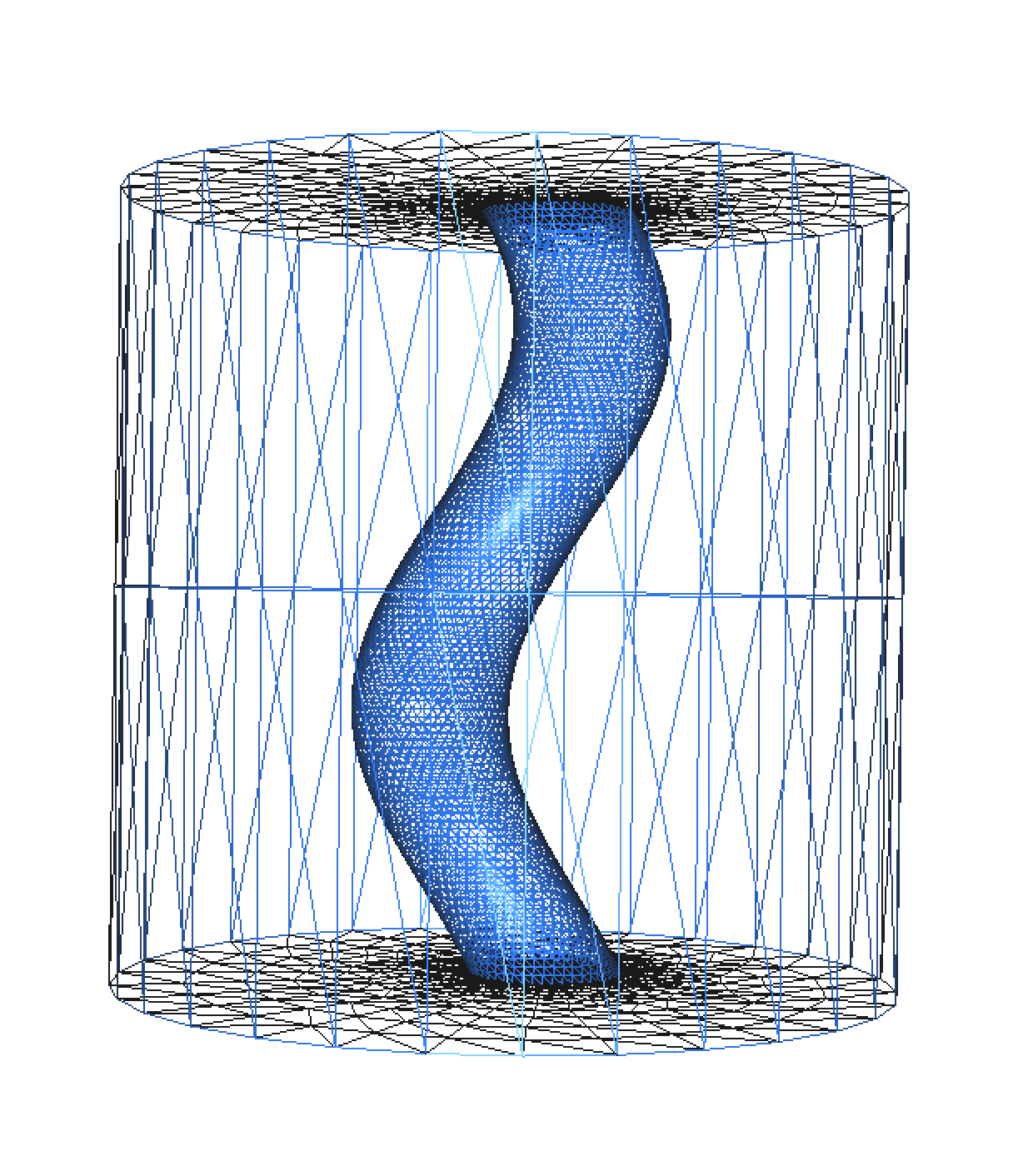}}
    \subfigure{\includegraphics[scale = .19]{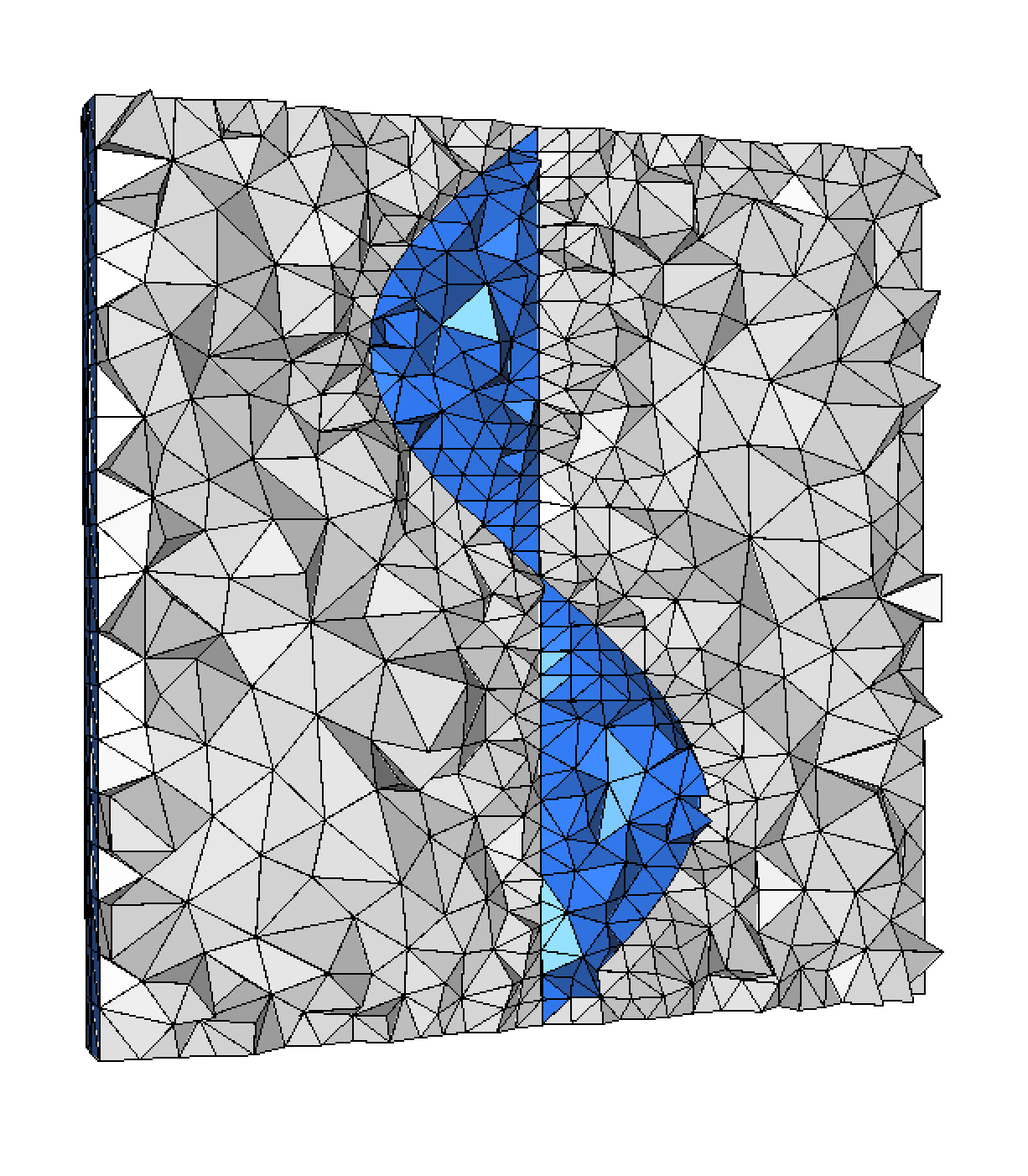}}
    \caption{The domain $\Omega$ in Example 2 (left) with its corresponding discretized space-time cylinder $Q_T$ (middle) and the cross-section of $Q_T$ along the $t$-axis (right).}
    \label{fig:model_3D}
\end{figure}

We set the diffusion coefficient $\kappa$ as $\left(\kappa_1, \kappa_2\right) = \left(2,1\right)$ and consider the desired state 
$$
u_d \left(\xb, t\right) = \left(x^2 + y^2 - 1\right) e^{2t} \qqqq \text{for } \left(\xb, t\right)\in Q_T,
$$
which is a smooth function within the entire cylinder $Q_T$. Our goal is to best approximate $u_d$ given the parameter $\eta = 10^{-2}$. Since the exact state and adjoint are unknown, we first compute a reference solution $\left(u_r, p_r\right)\in \Us_h\times \Us_h$ to Problem \eqref{eq: discrete coupled state-adjoint} on a highly refined mesh of $Q_T$, consisting of $404,958$ tetrahedra. The error is then assessed using the following metric
$$
\mathcal{E}_r := \left(\ \int\limits_{Q_T}\abs{\nabla\left(u_r - \overline{u}_h\right)}^2 + \abs{\nabla\left(p_r - \overline{p}_h\right)}^2\dx\dt\right)^{1/2}.
$$
The resulting error $\mathcal{E}_r$, along with its estimated convergence order, is presented in \Cref{tab:example2a}. Once again, we observe linear convergence, which confirms the error bounds in \Cref{cor: usual norm error estimate}.

\begin{table}[!htp]
    \centering
    \caption{Errors for various levels of mesh refinement in Example 2 with a smooth desired state.} 
    \label{tab:example2a}
    \begin{tabular}{ r c r c } 
        \toprule
        $\#$Dofs & Mesh size $h$  & Error $\mathcal{E}_r$& Order \\ \hline \hline
        450 &  $1.204 \times 10^{-1}$&$16.919$& -- \\ 
        1,928 &  $7.966 \times 10^{-2}$  &  $11.268$&1.002  \\ 
        4,623 & $6.034 \times 10^{-2}$  & $8.441$& 1.004 \\ 
        8,540 &  $4.917 \times 10^{-2}$ & $6.610$& 1.095  \\
        13,681& $4.167 \times 10^{-2}$  &$5.595$&0.914\\
        20,202 & $3.616 \times 10^{-2}$  & $4.703$&1.127\\
        \bottomrule
    \end{tabular}
\end{table}

Finally, we revisit this problem setting, but now with a discontinuous desired state
$$
u_d\left(\xb,t\right) = \begin{cases}
    2 & \text{if} \q x^2 + y^2 + \left(t-\dfrac{1}{2}\right)^2 \leq \dfrac{1}{5},\\
    0 & \text{else}.
\end{cases}
$$
Clearly, $u_d\notin \Hs^1\left(Q_T\right)$. Moreover, unlike $u^\ast$ and $p^\ast$, the discontinuity of $u_d$ occurs across the surface $x^2 + y^2 + \left(t-\frac{1}{2}\right)^2 = \frac{1}{5}$, rather than along the interface $\Gamma^\ast$. The discrete state $u_h$ with $h=6.674 \times 10^{-2}$ is shown in \Cref{fig:ex2}. As illustrated, $u_h$ approximates $u_d$ quite well, demonstrating the effectiveness of our scheme in discretizing Problem \eqref{eq: problem formulation}. Furthermore, we evaluate the convergence rate of $\mathcal{E}_r$ and summarize the results in \Cref{tab:example2b}. The observed convergence rate still aligns with the estimates presented in \Cref{theo: error estimate,cor: usual norm error estimate}. However, in this case, the convergence order is suboptimal. This behavior can be attributed to the nonsmooth nature of $u_d$, which probably limits the regularity of both $u^\ast$ and $p^\ast$ within the union of subdomains $Q_1\cup Q_2$. 

\begin{figure}[!tp]
    \centering
    \subfigure{\includegraphics[scale = .188]{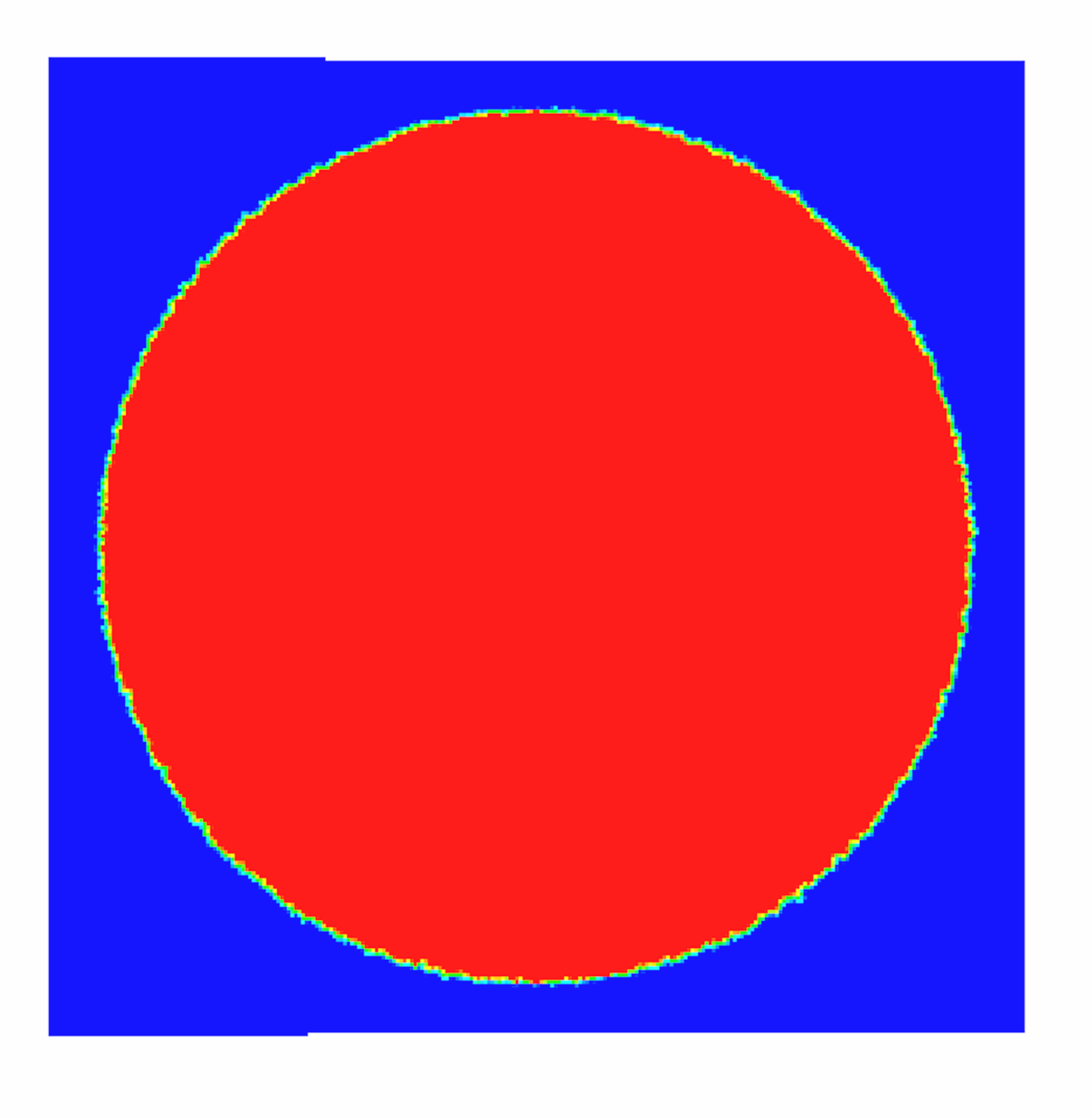}} \quad
    \subfigure{\includegraphics[scale = .25]{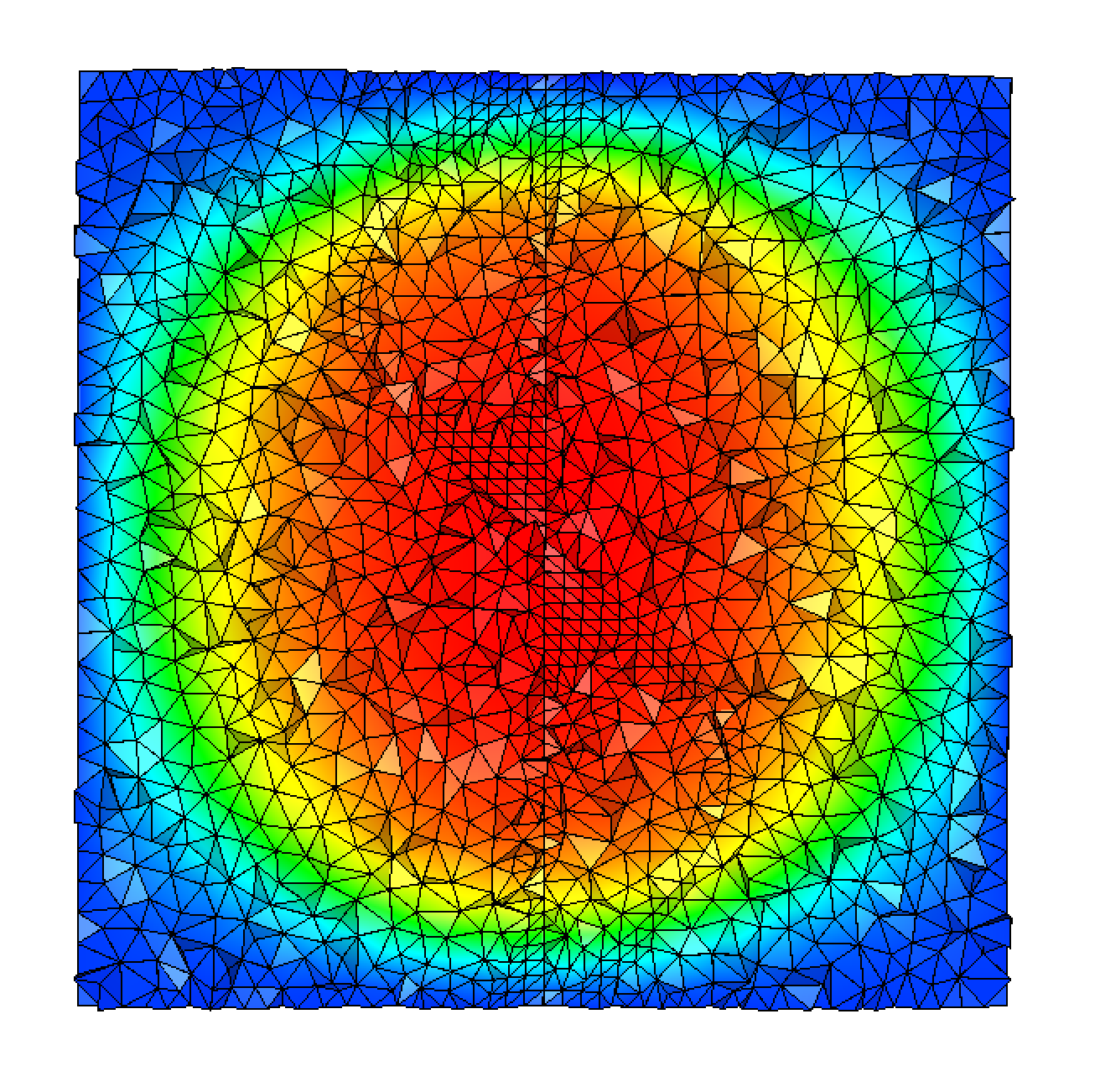}}
    \caption{The cross-section along the $t$-axis of the desired state $u_d$ (left) and the discrete state $u_h$ in Example 2 with $h=6.674 \times 10^{-2}$ (right).}
    \label{fig:ex2}
\end{figure}

\begin{table}[!htp]
    \centering
    \caption{Errors for various levels of mesh refinement in Example 2 with a discontinuous desired state.} 
    \label{tab:example2b}
    \begin{tabular}{ r c c c } 
        \toprule
        $\#$Dofs & Mesh size $h$ & Error $\mathcal{E}_r$& Order \\ \hline \hline
        450 &  $1.204 \times 10^{-1}$&$3.931$& -- \\ 
        1,928 &  $7.966 \times 10^{-2}$  &  $2.797$&0.838  \\ 
        4,623 & $6.034 \times 10^{-2}$  & $2.156$& 0.905 \\ 
        8,540 &  $4.917 \times 10^{-2}$ & $1.673$& 1.134  \\
        13,681& $4.167 \times 10^{-2}$ &$1.467 $&0.720\\
        20,202 & $3.616 \times 10^{-2}$  & $1.312 $&0.726\\
        \bottomrule
    \end{tabular}
\end{table}

%\section*{Acknowledgments}
%We would like to acknowledge the assistance of volunteers in putting
%together this example manuscript and supplement.

\bibliographystyle{siamplain}
\bibliography{SIAM_ref}
\end{document}

%% file: lvc_style.tex
\newcommand{\R}{\mathbb R}

\newcommand{\transpose}{\mathsf{T}}

\DeclareMathOperator{\Ls}{L}
\DeclareMathOperator{\LLs}{\mathbf L}
\DeclareMathOperator{\Hs}{H}

\DeclareMathOperator{\Ws}{W}

\DeclareMathOperator{\Cs}{C}
\DeclareMathOperator{\CCs}{\mathbf C}
\DeclareMathOperator{\Ds}{D}
\DeclareMathOperator{\Es}{E}

\DeclareMathOperator{\Us}{U}

\newcommand{\brac}[1]{\left\lbrace{#1}\right\rbrace}

\newcommand{\paren}[1]{\left({#1}\right)}
\newcommand{\norm}[1]{\left\lVert{#1}\right\rVert}

\newcommand{\abs}[1]{\left\vert{#1}\right\vert}

\newcommand{\inprod}[1]{\left\langle{#1}\right\rangle}

\newcommand{\ovl}[1]{\overline{#1}}

\newcommand{\vertiii}[1]{{\left\vert\kern-0.25ex\left\vert\kern-0.25ex\left\vert #1 
    \right\vert\kern-0.25ex\right\vert\kern-0.25ex\right\vert}}

\newcommand{\pa}{\partial}

\newcommand{\Gm}{\Gamma}

\newcommand{\Om}{\Omega}

\newcommand{\vphi}{\varphi}
\newcommand{\fa}{\forall}
\newcommand{\sst}{\subset}

\newcommand{\xb}{\boldsymbol{x}}

\newcommand{\vb}{\mathbf{v}}

\newcommand{\dx}{\,\mathrm{d}\xb}
\newcommand{\ds}{\,\mathrm{d}s}
\newcommand{\dt}{\,\mathrm{d}t}
\newcommand{\di}{\,\mathrm{d}}

\newcommand{\q}{\quad}
\newcommand{\qq}{\qquad}
\newcommand{\qqq}{\qquad\quad}
\newcommand{\qqqq}{\qquad\qquad}
\newcommand{\qqqqq}{\qquad\qquad\quad}